\documentclass[letter]{amsart}

\usepackage[utf8]{inputenc}
\usepackage[english]{babel}
\usepackage{amsmath,amssymb,amsthm,mathrsfs}
\usepackage{esint}
\usepackage[colorlinks,linktocpage]{hyperref}
\usepackage[all]{xy}
\usepackage{import}
\usepackage{pdfpages}
\usepackage{transparent}
\usepackage{xcolor}
\usepackage{comment}

\newtheorem{thm}{Theorem}[section]
\newtheorem{lem}[thm]{Lemma}

\newtheorem{prop}[thm]{Proposition}

\theoremstyle{definition}
\newtheorem{defn}{Definition}[section]

\theoremstyle{remark}
\newtheorem{rmk}{Remark}[section]
\newtheorem*{rmk*}{Remark}

\newtheorem*{fact*}{Fact}

\def\diam{\mathrm{diam}}

\def\Area{\mathrm{Area}}
\def\Length{\mathrm{Length}}
\def\Ric{\mathrm{Ric}}
\def\Vol{\mathrm{Vol}}

\title{Stability for the 3D Riemannian Penrose inequality}
\author{Conghan Dong}
\address{Mathematics Department, Stony Brook University, 100 Nicolls Rd, Stony Brook, NY 11794, USA}
\email{conghan.dong@duke.edu}
\begin{document}
\date{\today}

\maketitle

\begin{abstract}
We show that the Schwarzschild $3$-manifold is stable for the $3$-dimensional Riemannian Penrose inequality in the pointed measured Gromov-Hausdorff topology, modulo negligible domains and boundary area perturbations. 
\end{abstract}

\tableofcontents

\section{Introduction}

Let $(M^3, g)$ be a complete, smooth, asymptotically flat $3$-manifold with nonnegative scalar curvature and ADM mass $m(g)$ whose outermost minimal boundary has total surface area $A$. Then the (Riemannian) Penrose inequality states that $m(g) \geq \sqrt{\frac{A}{16 \pi }} $, with equality if and only if $(M^3, g)$ is isometric to the Schwarzschild $3$-manifold of mass $m(g)$. This theorem was proved by Huisken-Ilmanen \cite{HI'01} in the case that the boundary is connected, and by Bray \cite{Bray'01} in the general case. See also \cite{BrayLee'09,KWY17,AMMO22}, etc., for additional extensions and generalizations. As a corollary, the Penrose inequality implies the positive mass theorem that $m(g) \geq 0$, which was first proved by Schoen-Yau \cite{SY79}. See also other generalizations \cite{Witten81, SY'81}, etc.

A natural question to ask regarding the Penrose inequality is whether almost equality would imply that the manifold is close to the Schwarzschild manifold in some topology. This is called the stability problem for the Penrose inequality. Recently, there has been growing interest in studying similar stability problems. Notably, the author and Antoine Song proved a stability result for the positive mass theorem (see \cite{DS'23} and the references therein).

In this paper, we prove that, up to boundary area perturbations, the Schwarzschild manifold is stable for the Penrose inequality in the pointed measured Gromov-Hausdorff topology, modulo negligible domains (cf. Definition \ref{pm-GH mod}). More precisely, we have the following:

\begin{thm}\label{Penrose-stability}
	Let $A_0 \geq 0$ be a fixed constant and $(M_i^{3}, g_i)$ be a sequence of complete one-ended asymptotically flat $3$-manifolds, each of which has nonnegative scalar curvature, a compact outermost minimal boundary (possibly with multiple components) with total area $A_i \to A_0$. Suppose that the ADM mass $m(g_i)\to \sqrt{\frac{A_0}{16 \pi }} $. Then for all $i$, there is a connected closed subset $N_i \subset M_i$ containing the end, such that its boundary $\partial N_i = \Sigma _i^{\mathrm{a}} \cup \Sigma _i^{\mathrm{s}}$, where the major part $\Sigma _i^{\mathrm{a}}$ is connected and satisfies that 
	$$\Area_{g_i}(\Sigma _i^{\mathrm{a}}) \to A_0,$$
	and the minor part $\Sigma _i^{\mathrm{s}}$ satisfies
	$$\mathrm{Area}_{g_i}(\Sigma _i^{\mathrm{s}}) \to 0.$$
	Moreover, for any $p_i \in \Sigma _i^{\mathrm{a}}$, we have
	$$
	(N _i , \hat{d}_{g_i, N_i}, p_i ) \to (M_{\mathrm{Sch}}^{3}, g_{\mathrm{Sch}}, x_o)
	$$ 
	in the pointed measured Gromov-Hausdorff topology, and
	\begin{align*}
	(\Sigma _i^{\mathrm{a}}, \hat{d}_{g_i, \Sigma _i^{\mathrm{a}}} ) \to (\partial M_{\mathrm{Sch}}^3, \hat{d}_{g_{\mathrm{Sch}}, \partial M^3_{\mathrm{Sch}}})	
	\end{align*}
	in the measured Gromov-Hausdorff topology. Here, we denote by $(M_{\mathrm{Sch}}^{3}, g_{\mathrm{Sch}}, x_o)$ the standard Schwarzschild $3$-manifold with boundary area $\mathrm{Area}(\partial M^{3}_{\mathrm{Sch}}) = A_0$ and mass $m(g_{\mathrm{Sch}}) = \sqrt{ \frac{A_0}{16 \pi }} $. $x_o \in \partial M_{\mathrm{Sch}}^3$ is a base point, and $\hat{d}_{g_i, \cdot }$ are the length metrics on the corresponding spaces induced by $g_i$ (see Figure \ref{fig1} for a simplified geometric picture).
\end{thm}

We also have a similar stability result for the mass-capacity inequality. Recall that for an asymptotically flat $3$-manifold $(M^3,g)$ with outermost minimal boundary $\Sigma$ and an end $\infty_1$, the capacity of $\Sigma$ in $(M^3, g)$ is defined by $$
\mathcal{C}(\Sigma, g) := \inf \left\{ \frac{1}{\pi } \int_{M}|\nabla \varphi |^2 \mathrm{dvol}_g: \varphi \in C^{\infty}(M), \varphi =\frac{1}{2} \text{ on } \Sigma, \lim_{x\to \infty_1} \varphi (x) =1	\right\}. 
$$ 
Then, as a corollary of the positive mass theorem, it was shown by \cite[Theorem 9]{Bray'01} that $m(g)\geq  \mathcal{C}(\Sigma, g)$, and equality holds if and only if $(M^3, g)$ is the Schwarzschild $3$-manifold. For more details, please refer to Section \ref{capacity}.

\begin{thm}\label{mass-cap-stability}
	Let $m_0>0$ be a fixed constant and $(M_i^3, g_i)$ be a sequence of complete one-ended asymptotically flat $3$-manifolds, each of which has nonnegative scalar curvature and a compact connected outermost minimal boundary. Suppose that both $m(g_i) \to m_0$ and $  \mathcal{C}(\partial M_i, g_i) \to m_0$. Then for all $i$, there is a connected closed subset $E_i \subset M_i$ containing the end, such that its boundary $\partial E_i = \Sigma _i^{\mathrm{b}} \cup \Sigma _i^{\mathrm{s}}$, where the major part $\Sigma _i^{\mathrm{b}}$ is connected and satisfies that
	$$ \sup_{x \in \Sigma _i^{\mathrm{b}} } d_{g_i}(x, \partial M_i) \to 0, $$
	and the minor part $\Sigma _i^{\mathrm{s}}$ satisfies
	$$\Area_{g_i}(\Sigma^{\mathrm{s}}_i) \to 0.$$
	Moreover, for any $p_i \in \Sigma _i^{\mathrm{b}}$, we have
	$$
	(E_i, \hat{d}_{g_i, E_i}, p_i ) \to (M_{\mathrm{Sch}}^{3}, g_{\mathrm{Sch}}, x_o)
	$$ 
	in the pointed measured Gromov-Hausdorff topology, and 
	\begin{align*}
		(\Sigma _i^{\mathrm{b}}, \hat{d}_{g_i, \Sigma _i^{\mathrm{b}}}) \to (\partial M_{\mathrm{Sch}}^3, \hat{d}_{g_{\mathrm{Sch}}, \partial M_{\mathrm{Sch}}^3})	
	\end{align*}
	in the measured Gromov-Hausdorff topology. Here, we denote by $(M_{\mathrm{Sch}}^{3}, g_{\mathrm{Sch}}, x_o)$ the standard Schwarzschild $3$-manifold with boundary area $\mathrm{Area}(\partial M^{3}_{\mathrm{Sch}}) = 16 \pi m_0^2$ and mass $m(g_{\mathrm{Sch}}) = m_0 $. $x_o \in \partial M_{\mathrm{Sch}}^3$ is a base point, and $\hat{d}_{g_i, \cdot}$ are the length metrics on the corresponding spaces induced by $g_i$. 
\end{thm}

\begin{figure}[htpb]
	\centering
	\includegraphics[width=0.8\textwidth]{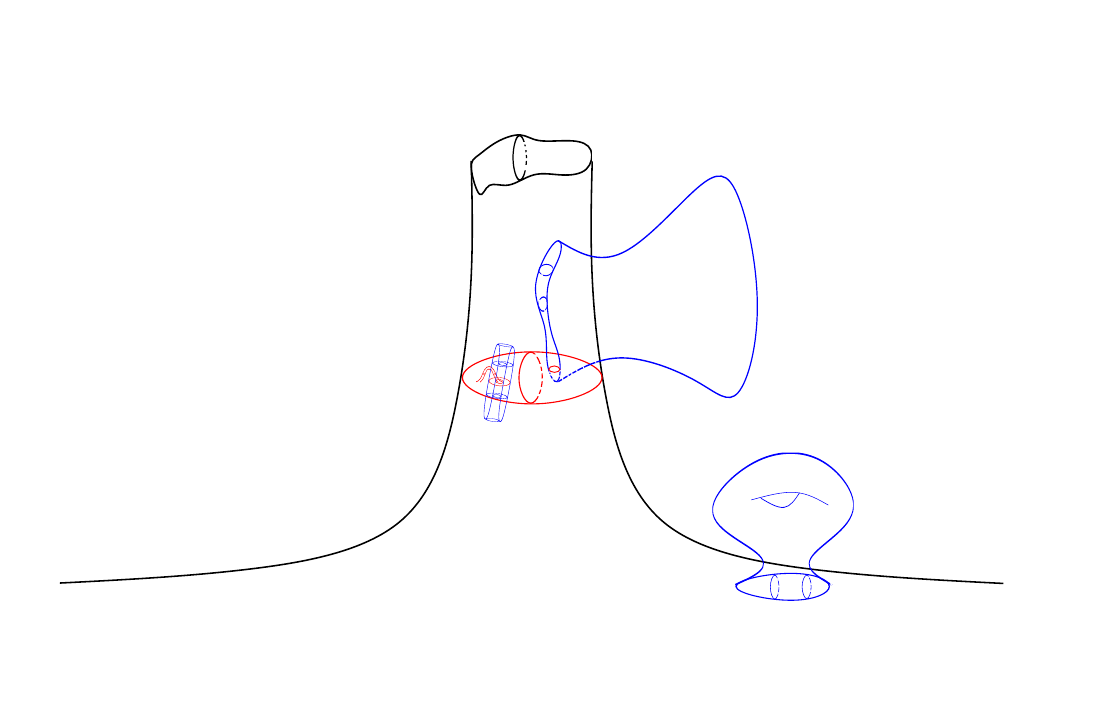}
	\caption{This simplified picture visually illustrates the conclusions of Theorem \ref{Penrose-stability}. The original manifold $M$ is represented in black and blue. $N$ refers to the region below the red part, excluding the blue part. $\Sigma ^{\mathrm{a}}$ corresponds to the red part, resembling an almost standard sphere with some minor handles attached and a few small disks or annuli removed, while $\Sigma ^{\mathrm{s}}$ is the boundary of the blue part, which has negligible area. In Theorem \ref{mass-cap-stability}, the long neck between the boundary of the manifold and the red part in this picture will not occur.} 
	\label{fig1}
\end{figure}

Our stability results can be viewed as generalizations of the main theorem in \cite{DS'23}. In particular, the case when $A_0=0$ in Theorem \ref{Penrose-stability} was proved in \cite{DS'23}. Our proof of the main theorems will build upon the tools developed in \cite{DS'23}, as well as some new techniques.

We give a remark on the boundary area perturbations in Theorem \ref{Penrose-stability}. Notice that in Theorem \ref{mass-cap-stability}, we can ensure that $\Sigma _i^{\mathrm{b}}$, the major part of the boundary of $E_i$, lies in a small neighborhood of the original boundary $\partial M_i$, but in Theorem \ref{Penrose-stability}, this is not necessarily true in general. In \cite[Example 5.2]{LS12}, Lee-Sormani constructed manifolds with almost equality in the Penrose inequality that are close to Schwarzschild spaces with a cylinder of arbitrary length appended to the boundary. See also \cite{ManScho15}. In other words, we need to cut out the long necks in these examples such that the remaining part has almost the same boundary area and is stable for the Penrose inequality. Such phenomena are consistent with our stability results. 

It would also be an interesting question to prove similar results for other topologies. See, for example, \cite{LS12, LS14,LNN20, KKL21, Dong22}, etc., and the references therein for related results.

\subsection*{Outline of the proof}
We will mainly focus on the proof of Theorem \ref{mass-cap-stability}. Up to a boundary area perturbation, Theorem \ref{Penrose-stability} will follow as a corollary of Theorem \ref{mass-cap-stability} by using an argument from \cite{Bray'01}.

Assume the conditions stated in Theorem \ref{mass-cap-stability}, i.e., $(M^3, g)$ is a complete asymptotically flat $3$-manifold with nonnegative scalar curvature and a connected outermost minimal boundary $\Sigma$ such that $m(g)>0$ and $m(g)-  \mathcal{C}(\Sigma, g) \ll 1$. By employing a doubling technique and allowing for a perturbation as in \cite{Bray'01}, we can assume that $\bar{M} = M \cup _\Sigma M$ is a smooth asymptotically flat $3$-manifold with nonnegative scalar curvature and two ends $\infty_1, \infty_2$.
Then, the infimum in the definition of capacity is achieved by the Green's function $f$ defined on $\bar{M}$, which satisfies
\begin{align*}
	\Delta _g f = 0,\\
	\lim_{x\to \infty_1} f(x) =1,\\ 
	\lim_{x\to \infty_2} f(x) =0.   
\end{align*}
From symmetry of $\bar{M}$, $f$ 
equals $\frac{1}{2}$ on $\Sigma$. We now consider the conformal metric $h= f^{4} g$ on $\bar{M}$. Then, $\infty_2$ can be compactified such that $\bar{M}^* := \bar{M} \cup \{\infty_2\} $ is a smooth manifold, and $(\bar{M}^*, h)$ is an asymptotically flat $3$-manifold with nonnegative scalar curvature. It can also be shown that the ADM mass $m(h) = m(g) -  \mathcal{C}(\Sigma,g)$, as in \cite{Bray'01}.

Based on our assumptions, $m(h) \ll 1$, the case discussed in \cite{DS'23}. So, modulo negligible domains, $(\bar{M}^*, h)$ is close to the Euclidean $3$-space $\mathbb{R}^3$ in the pointed measured Gromov-Hausdorff topology. For the original metric $g = f^{-4} h$, we will show that $f$ is uniformly close to the conformal factor in the Schwarzschild metric in the following two steps.

We first prove a new integral inequality involving the scalar curvature and the Hessian of the Green's function (cf. Proposition \ref{u-good}). Let $u:= \frac{2}{\mathcal{C}(\Sigma ,g)}\cdot \frac{1-f}{f}$, then
\begin{align*}
	\begin{split}
		1 - \frac{\mathcal{C}(\Sigma ,g)}{m(g)}
		\geq
		       \frac{\mathcal{C}(\Sigma ,g)}{96 \pi }\int_{(M, h)} \frac{\left|\nabla ^2 u + |\nabla u|_h^{\frac{3}{2}} ( h - 3 \nu \otimes \nu )\right|_h^2}{|\nabla u|_h } + R_h |\nabla u|_h  \mathrm{dvol}_h,
	\end{split}
\end{align*}
where $\nu = \frac{\nabla u}{|\nabla u|_h}$ and the integration is taken over regular set of $u$. 
This formula is very similar to the mass inequality proved by \cite[Theorem 1.2]{BKKS'22}.
One can follow the proof of \cite[Theorem 1.2]{BKKS'22} and employ the technique of 
integration over level sets of $u$ to control the integral in the inequality above. 
This method was initially explored by \cite{Stern'22}, 
and there have been many other applications and generalizations in recent studies, including \cite{BHKKZ23, AMO21, AMMO22}, and others. 
In the proof, we will first use this method to obtain a preliminary integral inequality in Proposition \ref{f-prop}. Then, using a corollary of that inequality, we can finally derive the desired integral inequality as stated above. As a corollary, we also provide another proof of the mass-area-capacity inequality (cf. Theorem \ref{mass-area-capacity}), which was first proved in \cite{BM08}.

Then, by using the integral inequality, together with the techniques used in \cite{DS'23}, we are able to find a region $\mathcal{E} \subset \bar{M}^*$ with a small-area boundary. In this region $\mathcal{E}$, the behavior of $|\nabla f|$ closely resembles the conformal factor in the Schwarzschild metric, and particularly, $|\nabla f|$ is uniformly bounded. By the Arzel\`{a}-Ascoli theorem, along the convergence of $(\mathcal{E}, h)$, we can take the limit of such $f$ and obtain a limit function $f_\infty$ defined on $\mathbb{R}^3$. However, it is not immediately clear whether $f_\infty$ is precisely the conformal factor in the Schwarzschild metric. 
To address this, we note that a favorable property of the pointed measured Gromov-Hausdorff convergence modulo negligible domains is that $f$ also converges to $f_\infty$ in the $W^{1,2}$-sense (cf. Lemma \ref{weak-W12}). Therefore, the elliptic equations satisfied by $f$ are preserved in this convergence, and $f_\infty$ also satisfies an elliptic equation on $\mathbb{R}^3$. The fact that the area of the boundary $\partial \mathcal{E}$ converges to $0$ has been essential here. The elliptic equation satisfied by $f_\infty$, together with the control on the gradient, would imply the rigidity of $f_\infty$ (cf. Section \ref{conv-equations}). 

Finally, using metric geometry tools and similar arguments from \cite[Section 4]{DS'23}, we can apply these properties of the functions $f$ and metrics $h$ to prove the main theorems.

\subsection*{Acknowledgements}
I would like to thank Marcus Khuri, Christos Mantoulidis, Daniel Stern and Antoine Song for helpful discussions and comments. I thank Hubert Bray and Andre Neves for their interests. I also thank the referee for their helpful comments and suggestions that improved this paper.

This paper was partially supported by Simons Foundation International, LTD.

\section{Preliminaries}

\subsection{Notations}
We will use $C, C'$ to denote a universal positive constant (which may be different from line to line); $\Psi (t), \Psi (t|a,b,\ldots)$ denote small constants depending on $a, b, \ldots$ and satisfying 
$$
\lim_{t\to 0} \Psi (t)=0,\ \lim_{t\to 0} \Psi (t|a,b\ldots) =0,
$$ 
for each fixed $a,b,\ldots$

We denote the Euclidean metric by $g_{\mathrm{Eucl}}$ or $\delta $, and the induced geometric quantities with subindex $\mathrm{Eucl}$ or $\delta $.

For a general Riemannian manifold $(M,g)$ and any $\pi \in M$, the geodesic ball with center $p$ and radius $r$ is denoted by $B_g(p,r)$ or $B(p,r)$ if the underlying metric is clear. Given a Riemannian metric, for a surface $\Sigma$ and a domain $\Omega $, $\Area(\Sigma)$ is the area of $\Sigma$, $\Vol(\Omega )$ is the volume of $\Omega $ with respect to the metric.

Finally we introduce some notations about length metric that will be used later.
Given a subset $U$ in a Riemannian manifold $(M,g)$, let $(U, \hat{d}_{g,U})$ be the induced length metric on $U$ of the metric $g$, that is, for any $x_1, x_2 \in U$,
$$
\hat{d}_{g,U}(x_1, x_2):= \inf \{L_g(\gamma ): \gamma \text{ is a rectifiable curve connecting } x_1, x_2 \text{ and } \gamma \subset U\} ,
$$ 
where $L_g(\gamma ) = \int_{0}^{1}|\gamma '|_g$ is the length of $\gamma $ with respect to metric $g$. By convention, two
points in two different path connected components of $U$ are at infinite $\hat{d}_{g, U}$-distance. 

For any $D>0$ and $p \in U$, we use $\hat{B}_{g,U}(p,D)$ to denote the geodesic ball inside $(U, \hat{d}_{g,U})$, that is
$$
\hat{B}_{g, U}(p,D) := \{ x \in U: \hat{d}_{g,U}(p,x) \leq D\} .
$$

\subsection{Geometry of Schwarzschild metric}

For a positve number $m>0$, the Schwarzschild $3$-manifold $(M^3_{\mathrm{Sch}}, g_{\mathrm{Sch}})$ (with mass $m$) is given by the following warped product metric on $\mathbb{S}^2 \times [0, \infty)$:
\begin{align}\label{sc-r}
g_{\mathrm{Sch}}= d s^2 + u_m(s)^2 g_{\mathbb{S}^2}, \ \ s \in [0, \infty),
\end{align} 
where $g_{\mathbb{S}^2}$ is the spherical metric with $\mathrm{Area}(\mathbb{S}^2, g_{\mathbb{S}^2}) = 4 \pi$, and $u_m$ is a positive increasing function satisfying
\begin{align}\label{u_m}
	u_m(0) = 2m,\ u_m'(0) =0,\ u_m'(s) = \left( 1- \frac{2m}{u_m(s)} \right) ^{\frac{1}{2}},\ u_m''(s) = \frac{m}{u_m(s)^2}.
\end{align}
Then the scalar curvature of $g_{\mathrm{Sch}}$ is identically zero, and the boundary $\Sigma_{\mathrm{Sch}}:= \partial M^3_{\mathrm{Sch}}$ is the only minimal surface inside $M^3_{\mathrm{Sch}}$.

Under Cartesian coordinate, $M^3_{\mathrm{Sch}}$ is diffeomorphic to $\mathbb{R}^3 \setminus B(\frac{m}{2})$, where $B(\frac{m}{2})$ is the Euclidean ball with radius $\frac{m}{2}$ around the center, and we have 
\begin{align}\label{sc-x}
	g_{\mathrm{Sch}, ij}(x) = \left( 1+ \frac{m}{2 |x|} \right) ^{4} \delta_{ij} ,\ \forall |x| \geq \frac{m}{2}.
\end{align}
This metric can also be extended to give a Schwarzschild metric $g_{\mathrm{Sch}}$ defined on $\mathbb{R}^3 \setminus \{0\} $.

Define $\rho _m(x ) := \mathrm{dist}(x , \Sigma_{\mathrm{Sch}}) $. Then 
$$
\rho _m(x) = \int_{\frac{m}{2}}^{|x|}\left( 1+ \frac{m}{2t} \right) ^2 dt = |x| - \frac{m^2}{4|x|} + m \log \frac{2|x|}{m},
$$ 
which implies that
\begin{align}\label{rho-vs-x}
\rho _m(x) - m \log \frac{2\rho _m(x)}{m} \leq |x| \leq \rho _m(x).
\end{align}

Using these two representations (\ref{sc-r}) and (\ref{sc-x}) of $g_{\mathrm{Sch}}$ to compute the area of geodesic spheres, we have 
$$
u_m(\rho _m(x) ) ^2 \cdot 4\pi = \left( 1+ \frac{m}{2|x|} \right) ^{4} \cdot 4\pi |x|^2,
$$ 
i.e.
\begin{align}\label{u-vs-x}
u_m(\rho _m(x) ) = \left( 1+ \frac{m}{2|x|} \right) ^2\cdot |x|.
\end{align} 
In particular, together with (\ref{rho-vs-x}),
\begin{align}
\lim_{r\to \infty} \frac{u_m(r)}{r} =1.
\end{align}

Now we introduce another harmonic function $f_m$ and rewrite above identities using $f_m$ instead of $|x|$. Define
\begin{align*}
	f_m(x) := \left( 1+ \frac{m}{2|x|} \right) ^{-1}.
\end{align*}
Standard computations imply that 
$$
\Delta _{g_{\mathrm{Sch}}} f_m =0,\ \ f_m = \frac{1}{2} \text{ on } \Sigma_{\mathrm{Sch}},\ \ \lim_{|x|\to \infty}f_m(x)= 1. 
$$ 
Then
\begin{align*}
	|x| = \frac{m}{2}\cdot \frac{f_m(x)}{1- f_m(x)},
\end{align*}
and
\begin{align}\label{rho-f}
	\rho _m(x) = \rho _m(f_m(x) ) = \frac{m}{2} \left( \frac{1}{1-f_m(x)} - \frac{1}{f_m(x)} \right) + m \log \frac{f_m(x)}{1- f_m(x)}.
\end{align}
Moreover,
\begin{align}\label{u-f}
	u_m(\rho _m(x) ) = \frac{m}{2}\cdot \frac{1}{f_m(x) (1- f_m(x) )}.
\end{align}

\subsection{Asymptotically flat $3$-manifolds}
A smooth orientable connected complete Riemannian $3$-manifold $(M^3, g)$ is called asymptotically flat if there exists a compact subset $K \subset M$ such that $M\setminus K= \sqcup_{k=1}^{N} M_{\mathrm{end}}^k$ consists of finite pairwise disjoint ends, and for each $1\leq k \leq N$, there exist $C>0, \sigma > \frac{1}{2}$, and a $C^\infty$-diffeomorphism $\Phi_k : M_{\mathrm{end}}^k \to  \mathbb{R}^3 \setminus B(1)$ such that under this identification, 
	$$
	| \partial ^{l}(g_{ij}- \delta _{ij})(x)| \leq C|x|^{-\sigma - |l|},
	$$ for all multi-indices $|l|=0,1,2$ and any $x \in \mathbb{R}^3\setminus B(1)$.
	Furthermore, we always assume the scalar curvature $R_g$ is integrable over $(M^3,g)$. The ADM mass from general relativity of each end $M_{\mathrm{end}}^{k}$, $1\leq k \leq N$, is then well-defined (see \cite{ADM'61, Bartnik'86}) and given by
	$$m_k(g):= \lim_{r\to \infty} \frac{1}{16\pi} \int_{S_r} \sum_{i,j} (g_{ij,i}-g_{ii,j}) \nu ^j dA$$
where $\nu $ is the unit outer normal to the coordinate sphere $S_r$ of radius $|x|=r$ in the given end, and $dA$ is its area element. 

\begin{defn}
	A surface $\Sigma \subset (M^3, g)$ is called a horizon if it is a minimal surface. It is called an outermost horizon if it is a horizon and it is not enclosed by another minimal surface in $(M^3, g)$.
\end{defn}

Let $(M^3, g)$ be an asymptotically flat  $3$-manifold. By Lemma 4.1 in \cite{HI'01}, we know that inside $M^3$, there is a trapped compact region $T$ whose topological boundary consists of smooth embedded minimal $2$-spheres. An ``exterior region'' $M^3_{ext}$ is defined as the metric completion of any connected component of $M \setminus T$ containing one end. Then $M^3_{ext}$ is connected, asymptotically flat, has a compact minimal boundary $\partial M^3_{ext}$ ($\partial M^3_{ext}$ may be empty), and contains no other compact minimal surfaces, that is, $M^3_{ext}$ is an asymptotically flat $3$-manifold with outermost horizon boundary.

We will be able to perturb an asymptotically flat metric to a metric with nicer behavior at infinity in each end because of the following definition and proposition.

\begin{defn}
	We say that $(M^3, g)$ is harmonically flat at infinity if $(M^3 \setminus K, g)$ is isometric to a finite disjoint union of regions with zero scalar curvature which are conformal to $(\mathbb{R}^3 \setminus B, \delta)$ for some compact set $K$ in $M^3$ and some ball $B$ in $\mathbb{R}^3$ centered around the origin.
\end{defn}
By definition, if $(M^3, g)$ is harmonically flat, then on each end, $g_{ij}(x)= V(x) \delta _{ij}$ for some bounded positive $\delta $-harmonic function $V(x)$, which satisfies that $\Delta _{\delta } V(x) =0$ and (c.f. \cite[Equation (10)]{Bray'01})
\begin{align}\label{V-expan}
	V(x) = a + \frac{b}{|x|} + O\left(\frac{1}{|x|^2}\right).
\end{align} 
In this case, its ADM mass on this end is given by $2ab$.

\begin{prop}[{\cite{SY'81}}]\label{hf-apprx}
	Let $(M^3,g)$ be a complete, asymptotically flat $3$-manifold with $R_g \geq 0$ and ADM mass $m_k(g)$ in the $k$-th end. For any $\epsilon >0$, there exists a metric $\hat{g}$ such that $e^{-\epsilon }g \leq \hat{g} \leq e^{\epsilon }g$, $R_{\hat{g}} \geq 0$, $(M^3, \hat{g})$ is harmonically flat at infinity, and $|m_k(\hat{g}) - m_k(g)| \leq \epsilon $, where $m_k(\hat{g})$ is the ADM mass of $\hat{g}$ in the $k$-th end.
\end{prop}

\subsection{pm-GH convergence modulo negligible domains}
In this subsection, we recall some definitions for the pointed measured Gromov-Hausdorff topology.

Assume $(X,d_X,x), (Y,d_Y, y)$ are two pointed metric spaces.
The pointed Gromov-Hausdorff (or pGH-) distance is defined in the following way. A pointed map $f:(X,d_X,x)\to (Y, d_Y,y)$ is called an $\varepsilon $-pointed Gromov-Hausdorff approximation (or $\varepsilon $-pGH approximation) if it satisfies the following conditions:
\begin{itemize}
	\item[(1)] $f(x)=y$;
	\item [(2)] $B(y, \frac{1}{\varepsilon }) \subset B_\varepsilon (f(B(x, \frac{1}{\varepsilon })) )$;
	\item[(3)] $|d_X(x_1, x_2) - d_Y(f(x_1), f(x_2) )|<\varepsilon $ for all $x_1, x_2 \in B(x, \frac{1}{\varepsilon })$.
\end{itemize}
The pGH-distance is defined by
\begin{align*}
	d_{pGH}( (X,d_X, x), & (Y,d_Y,y) ) :=\\
&\inf \{\varepsilon>0 : \exists\ \varepsilon\text{-pGH approximation } f:(X,d_X,x)\to (Y,d_Y,y)\} .
\end{align*} 

We say that a sequence of pointed metric spaces $(X_i, d_i, p_i)$ converges to a pointed metric space $(X, d, p)$ in the pointed Gromov-Hausdorff topology, if the following holds 
$$d_{pGH}( (X_i, d_i, p_i), (X, d, p) ) \to 0.$$

If $(X_i, d_i)$ are length metric spaces, i.e.  for any two points $x, y \in X_i$, $$
d_i(x,y)= \inf \{L_{d_i}(\gamma ): \gamma \text{ is a rectifiable curve connecting }x, y\}, 
$$ where $L_{d_i}(\gamma )$ is the length of $\gamma $ induced by the metric $d_i$, then equivalently, 
$$d_{pGH}( (X_i, d_i, p_i), (X,d,p) )\to 0$$ 
if and only if for all $D>0$, $$d_{pGH}( (B(p_i, D), d_i), (B(p, D), d) ) \to 0,$$ where $B(p_i, D)$ are the geodesic balls of metric $d_i$.

A pointed metric measure space is a structure $(X,d_X, \mu, x )$ where $(X,d_X)$ is a complete separable metric space, $\mu $ a Radon measure on $X$ and $x \in \mathrm{supp} (\mu)$. 

We say that a sequence of pointed metric measure length spaces $(X_i, d_i, \mu _i, p_i)$ converges to a pointed metric measure length space $(X, d, \mu , p)$ in the pointed measured Gromov-Hausdorff (or pm-GH) topology, if for any $\varepsilon >0, D>0$, there exists $N(\varepsilon ,D) \in \mathbb{Z}_+$ such that for all $i \geq N(\varepsilon , D)$, there exists a Borel $\varepsilon $-pGH approximation 
$$f_i^{D, \varepsilon }: (B(p_i, D), d_i, p_i) \to (B(p, D+\varepsilon ), d, p) $$
satisfying 
$$
(f_i^{D, \varepsilon })_{\sharp}(\mu _i|_{B(p_i, D)}) \text{ weakly converges to } \mu |_{B(p, D)}  \text{ as } i\to \infty, \text{ for }a.e. D>0.
$$ 

In the case when $X_i$ is an $n$-dimensional manifold, without extra explanations, we will always consider $(X_i, d_i, p_i)$ as a pointed metric measure space equipped with the $n$-dimensional Hausdorff measure $\mathcal{H}^n_{d_i}$ induced by $d_i$.

Finally, we introduce a notation about the topology used in this paper. For simplicity, we only consider manifolds, but one can easily generalize it to general metric measured spaces.
\begin{defn}\label{pm-GH mod}
For a sequence of pointed Riemannian $n$-manifolds $(M_i^n, g_i, p_i)$ and $(M^n, g, p)$, we say that $(M_i, g_i, p_i)$ converges to $(M, g, p)$ in the pointed measured Gromov-Hausdorff topology modulo negligible domains if there exist open subsets $Z_i \subset M_i$ such that $\mathcal{H}^{n-1}(\partial Z_i) \to 0$, $p_i \in M_i \setminus Z_i$ and 
$$
(M_i \setminus Z_i, \hat{d}_{g_i}, p_i) \to (M, d_g, p)
$$ 
in the pointed measured Gromov-Hausdorff topology for the induced length metric.
\end{defn}

We have the following theorem, which states that $C^0$-convergence of metric tensors modulo negligible domains implies Gromov-Hausdorff convergence of length metrics modulo negligible domains after peturbation.
\begin{thm}\label{GH-spikes}
	Assume that $n \geq 2$, and $(E_i^n, g_i)$ is a sequence of $n$-dimensional complete Riemannian manifolds with compact boundaries. If $(E_i, g_i)$ converges to $\mathbb{R}^n$ in the $C^0$-sense modulo negligible domains, that is, there exist embeddings $\mathbf{u}_i: E_i \to \mathbb{R}^n$ such that 
\begin{itemize}
	\item $\mathbf{u}_i(E_i)$ contains the end of $\mathbb{R}^n$,
	\item $\|(\mathbf{u}_{i}^{-1})^* g_i - g_{\mathrm{Eucl}}\|_{C^0} \to 0$,
	\item $\mathcal{H}^{n-1}(\partial E_i) \to 0$, 
\end{itemize}
then there exist closed subsets $E_i'' \subset E_i$ with compact boundaries, such that for any base points $p_i \in E_i''$, $(E''_i, g_i, p_i)$ converges to $(\mathbb{R}^n, g_{\mathrm{Eucl}}, 0)$ in the pointed measured Gromov-Hausdorff topology modulo negligible domains in the sense of Definition \ref{pm-GH mod}.
\end{thm}
Notice that this theorem was proved for dimension $3$ in \cite{DS'23}.
Using the same techniques as in \cite{DS'23}, along with inductive arguments, we can extend this result to a general dimensional version. For the reader's convenience, we provide a detailed proof in Appendix \ref{appendix}.

\section{Capacity of horizon and Green's function}\label{capacity}
In this section, we introduce some properties of capacity and Green's function, which are known in the literature and will be used later in this paper. Most of this section follows from \cite{Bray'01}. 

Let's firstly introduce the capacity of a surface in the special case when it is the horizon of an asymptotically flat $3$-manifold which is harmonically flat at infinity. 
\begin{defn}
	Given a complete, asymptotically flat $3$-manifold $(M^3, g)$ with a connected outermost horizon boundary $\Sigma$, nonnegative scalar curvature and one asymptotically flat end $\infty_1$, the capacity of $\Sigma$ in $(M^3,g)$ is defined by
	\begin{align*}
		\mathcal{C}(\Sigma, g) := \inf \left\{ \frac{1}{ \pi } \int_{M^{3}}|\nabla \varphi |^2 \mathrm{dvol}_g: \varphi \in C^{\infty}(M),\ \varphi = \frac{1}{2} \text{ on } \Sigma,\ \lim_{x \to \infty_1}\varphi (x) =1 \right\} .
	\end{align*}
\end{defn}
From standard theory (c.f. \cite{Bartnik'86}), the infimum in the definition of $\mathcal{C}(\Sigma, g)$ is achieved by the Green’s function $\varphi \in C^{\infty}(M^3) $ which satisfies
\begin{align}\label{Green-1end}
	\begin{split}
		\Delta_g \varphi &=0,\\
		\varphi &= \frac{1}{2} \text{ on } \Sigma,\\
		\lim_{x \to \infty_1} \varphi (x) &=1.
	\end{split}
\end{align}
By maximum principle, $\varphi(x) \in [\frac{1}{2}, 1)$ for any $x \in M^3$.
Define the level sets of $\varphi $ to be
$$
\Sigma^{\varphi }_t := \{x \in M^3: \varphi (x) = t\} .
$$ 
Then by Sard's theorem, $\Sigma^{\varphi }_t$ is a smooth surface for almost all $t \in (\frac{1}{2}, 1)$. By the co-area formula, 
\begin{align*}
	\mathcal{C}(\Sigma, g) &= \frac{1}{ \pi } \int_{\frac{1}{2}}^{1} \int_{\Sigma^{\varphi }_t} |\nabla \varphi |.
\end{align*}
For any regular value $t \in (\frac{1}{2}, 1)$, integrating $\Delta \varphi =0$ over $\{\frac{1}{2} \leq \varphi \leq t\} $, and using Stokes' theorem, we have 
\begin{align}\label{h(ooo)}
	\int_{\Sigma} |\nabla \varphi | = \int_{\Sigma^{\varphi }_t} |\nabla \varphi |.
\end{align}
So
\begin{align}\label{cap-phi}
	\mathcal{C}(\Sigma, g) = \frac{1}{2 \pi } \int_{\Sigma}|\nabla \varphi |.
\end{align}
When $(M^3, g)$ is harmonically flat at infinity, we have the following expansion of the Green's function (c.f. \cite{Bartnik'86}, \cite[Equation (80)]{Bray'01}):
\begin{align}\label{phi-expan}
	\varphi (x) = 1 - \frac{\mathcal{C}(\Sigma,g)}{2|x|} + O\left(\frac{1}{|x|^2}\right) \text{ as } x\to \infty_1.
\end{align}

Now we introduce another definition which is closely related to the capacity of a horizon surface. 
\begin{defn}
	Given a complete, asymptotically flat $3$-manifold $(\bar{M}^3, \bar{g})$ with multiple asymptotically flat ends and one chosen end $\infty_1$, define
	\begin{align*}
		\mathcal{C}(\bar{g}) := \inf \left\{ \frac{1}{2 \pi }\int_{\bar{M}^3}|\nabla \phi |^2 \mathrm{dvol}_{\bar{g}}:  \phi \in \mathrm{Lip}(\bar{M}),\ \lim_{x\to \infty_1} \phi (x) =1,\ \lim_{x\to \{\infty_k\} _{k \geq 2}}\phi (x) =0 \right\} .
	\end{align*}
\end{defn}
Similarly the infimum in the definition of $\mathcal{C}(\bar{g})$ is achieved by the Green’s function $\phi $ which satisfies
\begin{align}\label{Green-2end}
	\begin{split}
		\Delta_{\bar{g}} \phi & =0,\\
		\lim_{x\to \infty_1} \phi (x) &= 1,\\
		\lim_{x\to \infty_k} \phi (x) &=0 \text{ for all } k \geq 2.
	\end{split}
\end{align}

For a complete asymptotically flat $3$-manifold $(M^3, g)$ with a compact outermost horizon boundary $\Sigma$, nonnegative scalar curvature and one end $\infty_1$, we can take another copy of $(M^3, g)$ and glue them together along the boundary $\Sigma$ to get a new metric space $(\bar{M}, \bar{g})$. In general, $(\bar{M}, \bar{g})$ is only a Lipschitz manifold with two asymptotically flat ends $\{\infty_1, \infty_2\} $. From the proof of \cite[Theorem 9]{Bray'01}, for any $\delta >0$ small enough, we can smooth out $(\bar{M}, \bar{g})$ and construct a smooth complete $3$-manifold $(\tilde{M}_{\delta }, \tilde{g}_\delta )$ with nonnegative scalar curvature and two asymptotically flat ends which, in the limit as $\delta \to 0$, approaches $(\bar{M}, \bar{g})$ uniformly. For reader's convenience, we recall the details of \cite{Bray'01} in the following.

Let $(M_1^3, g), (M_2^3,g)$ be the two copies of $(M^3, g)$. A first step is to construct a smooth manifold (c.f. \cite[Equation (92)]{Bray'01})
\begin{align*}
	(\tilde{M}_\delta, \bar{g}_\delta) : = (M_1^{3}, g) \sqcup (\Sigma \times (0, 2 \delta ), G) \sqcup (M_2^3, g),
\end{align*}
where $\Sigma \times \{0\} $ and $\Sigma \times \{2\delta \} $ are identified with $\Sigma \subset (M^3, g)$, $G$ is a warped product metric and symmetric about $t = \delta $, and $\Sigma \times \{\delta \} \subset (\Sigma \times (0, 2\delta ), G)$ is totally geodesic. In general, the scalar curvature of $G$ only satisfies $R_G \geq R_0$ for some constant $R_0 \leq 0$ independent of $\delta $, and may not be nonnegative. 

Then a second step is to take a conformal deformation of $\bar{g}_\delta $ to get a new metric with nonnegative scalar curvature. Define a smooth function $\mathcal{R}_\delta $, which equals $R_0$ in $\Sigma \times [0, 2\delta ]$, equals $0$ for $x$ more than a distance $\delta $ from $\Sigma \times [0, 2\delta ]$, takes values in $[R_0, 0]$ everywhere and symmetric about $\Sigma \times \{\delta \} $. In particular, $R_{\bar{g}_\delta }(x) \geq \mathcal{R}_\delta (x)$ for any $x \in \tilde{M}_\delta $. Define $u_\delta (x)$ such that (c.f. \cite[Equation (101)]{Bray'01})
\begin{align}\label{u_delta}
	\begin{split}
	(-8 \Delta _{\bar{g}_\delta } + \mathcal{R}_\delta (x) ) u_\delta (x) &=0, \\
	\lim_{x\to \{\infty_1, \infty_2\} } u_\delta (x) &=1.
\end{split}
\end{align}
Then $u_\delta $ is a smooth function and satisfies that (c.f. \cite[Equation (102)]{Bray'01})
\begin{align*}
	1 \leq u_\delta (x) \leq 1 + \epsilon (\delta )
\end{align*}
where $\epsilon $ goes to $0$ as $\delta \to 0$. Define 
\begin{align*}
	\tilde{g}_\delta := u_\delta ^{4} \cdot \bar{g}_\delta.
\end{align*}
The scalar curvature of $\tilde{g}_\delta $ satisfies 
\begin{align*}
	R_{\tilde{g}_\delta }&= u_\delta ^{-5} \left( R_{\bar{g}_\delta } u_\delta  - 8 \Delta _{\bar{g}_\delta } u_\delta  \right) \\
	& = u_\delta ^{-4} \left( R_{\bar{g}_\delta } - \mathcal{R}_\delta  \right) \\
	& \geq 0.
\end{align*}
By definition, $\lim_{\delta \to 0} m(\tilde{g}_\delta ) = m(\bar{g})$ and $\lim_{\delta \to 0} \mathcal{C}(\tilde{g}_\delta) = \mathcal{C}(\bar{g})$. 

To see the relation between $\mathcal{C}(\bar{g})$ and $\mathcal{C}(\Sigma, g)$, we define the reflection map 
$$
\Phi : M_1^3 \cup _{\Sigma} M_2^3 \to M_1^3 \cup _{\Sigma} M_2^3
$$ 
such that for any $x \in M_1^3$, $\Phi (x) \in M_2^3$ is the same point under the identification $M_1^3 = M_2^3 = M^3$, $\Phi ^2 = \mathrm{Id}$ and $\Phi |_{\Sigma} = \mathrm{Id}$. If $\phi $ satisfies (\ref{Green-2end}), then $1- \phi \circ \Phi $ also satisfies (\ref{Green-2end}) and by the uniqueness we have $\phi (x) = 1- \phi \circ \Phi (x)$, which implies that $\phi |_{\Sigma} = \frac{1}{2}$. So $\phi |_{M_{1}^3} $ also satisfies (\ref{Green-1end}), which implies that 
\begin{align}
\mathcal{C}(\bar{g}) =  \mathcal{C}(\Sigma, g).
\end{align}

Similarly, we can define the reflection map $\Phi _\delta : \tilde{M}_\delta \to \tilde{M}_\delta $ and from the equation (\ref{u_delta}) and the fact that $\bar{g}_\delta = \bar{g}_\delta \circ \Phi _\delta, \mathcal{R}_\delta = \mathcal{R}_\delta \circ \Phi _\delta  $, we know $u_\delta $ is also symmetric about $\Sigma \times \{\delta \} $ and particularly $\left<\nabla u_\delta , \vec{n} \right>_{\bar{g}_\delta } =0$ on $\Sigma \times \{\delta \} $, where $\vec{n}$ is the normal vector of $\Sigma \times \{\delta \} \subset (\tilde{M}_\delta , \bar{g}_\delta )$. Thus, $\tilde{g}_\delta $ is symmetric about $\Sigma \times \{\delta \} $ and the mean curvature of $\Sigma \times \{\delta \} \subset (\tilde{M}_\delta , \tilde{g}_\delta )$ is
\begin{align*}
	H_{(\Sigma \times \{\delta \} , \tilde{g}_\delta )} = u_\delta ^{-2} H_{(\Sigma \times \{\delta \} , \bar{g}_\delta )} - 2 \left< \nabla u_\delta ^{-2}, \vec{n} \right>_{\bar{g}_\delta } = 0.
\end{align*}
Let $(M_{\delta}, \tilde{g}_\delta )$ be one half of $(\tilde{M}_{\delta }, \tilde{g}_\delta )$ with minimal boundary $\Sigma_{\delta }:=\Sigma \times \{\delta \} $ and one asymptotically flat end. Then $(M_{\delta}, \tilde{g}_\delta, \Sigma_\delta )$ converges to $(M^3, g, \Sigma)$ uniformly as $\delta \to 0$. 

Without loss of generality, by applying Proposition \ref{hf-apprx}, we can assume that $(\tilde{M}_{\delta }, \tilde{g}_\delta )$ is also harmonically flat at infinity.

In summary, we have the following proposition. See more details in \cite{Bray97}.
\begin{prop}\label{doubling-approx}
	Given a complete one-ended asymptotically flat $3$-manifold $(M^3, g)$ with a connected outermost horizon boundary $\Sigma$ and nonnegative scalar curvature, there is a sequence of smooth complete $3$-manifolds $(\tilde{M}^3_\delta , \tilde{g}_\delta )$, which have nonnegative scalar curvature and two harmonically flat ends, and are symmetric about a minimal surface $\Sigma_{\delta } \subset (\tilde{M}^{3}_\delta , \tilde{g}_\delta )$, such that $(\tilde{M}_\delta ,\tilde{g}_\delta) \to (\bar{M}, \bar{g})$ and $(M_\delta , \tilde{g}_\delta ) \to (M, g)$ uniformly as $\delta \to 0$, where $(\bar{M}, \bar{g})$ is the doubling of $(M, g)$ along the boundary $\Sigma$, and $(M_\delta , \tilde{g}_\delta )$ is one half of $(\tilde{M}_\delta , \tilde{g}_\delta )$ with minimal boundary $\Sigma_\delta $.
\end{prop}

To conclude this section, we give a remark about the relations between the mass, capacity, and boundary area of the outermost horizon by briefly recalling Bray's proof of the Penrose inequality in \cite{Bray'01}. Given a complete smooth $3$-manifold $(M^3, g_0)$ with a harmonically flat end, nonnegative scalar curvature, an outermost minimizing horizon $\Sigma_0$ of total area $A_0$ and total mass $m_0$. Then for all $t \geq 0$, we can construct a continuous family of conformal metrics $g_t$ on $M^3$ which are asymptotically flat with nonnegative scalar curvature and total mass $m(t)$. Let $\Sigma(t)$ be the outermost minimal enclosure of $\Sigma_0$ in $(M^3, g_t)$, and $M(t)$ the asymptotically flat manifold with boundary $\Sigma (t)$. Then $\Sigma(t)$ is a smooth outermost horizon in $(M(t), g_t)$ with area $A(t)$ being a constant function about $t$. It was shown that $m(t)$ is decreasing, and as $t\to \infty$, $(M(t), g_t)$ approaches a Schwarzschild manifold $(\mathbb{R}^3\setminus \{0\} , g_{\mathrm{Sch}})$ with total mass $\lim_{t\to \infty}m(t) = \sqrt{\frac{A_0}{16 \pi }} .$ In particular, $m_0 \geq \sqrt{\frac{A_0}{16 \pi }} $, which proves the Penrose inequality.

\section{Integral estimate of the Hessian of the Green's function}
In this and the following section, we assume that $(\tilde{M}^{3}, g)$ is a complete, asymptotically flat $3$-manifold with two harmonically flat ends $\{\infty_1, \infty_2\} $ and nonnegative scalar curvature obtained as in Proposition \ref{doubling-approx}. In particular, the topology of $\tilde{M}^3$ is $\mathbb{R}^3 \setminus \{0\} $, and there is a minimal surface $\Sigma \subset (\tilde{M}^3, g)$ such that $g$ is symmetric about $\Sigma$. Let $(M^3, g)$ be the half of $(\tilde{M}^3, g)$ which contains the end $\infty_1$ and has minimal boundary $\Sigma$. So $\Sigma$ is diffeomorphic to a $2$-sphere, $M^3$ is diffeomorphic to $\mathbb{R}^3 \setminus B(1)$, and $(\tilde{M}^3, g) = (M^3, g) \cup _{\Sigma} (M^{3,\prime}, g)$, where we use $M^{\prime}$ to denote a copy of $M$ containing the other end $\infty_2$.

Let $f(x)$ be the solution to (\ref{Green-2end}) on $(\tilde{M}^3, g)$, that is
\begin{align}\label{har-eq}
	\begin{split}
	\Delta_{g} f &=0,\\
	\lim_{x \to \infty_1} f(x) &= 1,\\
	\lim_{x\to \infty_2} f(x) &= 0.
\end{split}
\end{align}
Then $f$ is a smooth function satisfying $0< f <1$ and the following expansion at infinity (c.f. \cite{Bartnik'86, Bray'01})
\begin{align}\label{f-eq2}
		\begin{split}
		f(x) &= 1 - \frac{c_1}{|x|} + O\left( \frac{1}{|x|^2} \right)  \text{ as } x \to \infty_1, \\
		f(x) &= \frac{c_2}{|x|} + O\left( \frac{1}{|x|^2} \right)  \text{ as } x\to \infty_2,
	\end{split}
	\end{align}
	where $c_k$ are positive constants for $k= 1, 2$. Moreover, for some $\tau \in (0,1)$,
	\begin{align}\label{f-eq3}
		\begin{split}
			\partial _j f(x) & = \frac{c_1}{|x|^2}\cdot  \frac{x^j}{|x|} + O\left(\frac{1}{|x|^{2+\tau  }}\right),\\
			\partial _j \partial _k f(x) &= \frac{c_1 \delta _{jk}}{|x|^3} - \frac{3c_1}{|x|^3}\cdot  \frac{x^j x^k}{|x|^2} + O\left( \frac{1}{|x|^{3+\tau  }} \right) .
		\end{split}
	\end{align}

	By the symmetry of $\tilde{g}$ about $\Sigma$, we know that on $(M^3, g)$, $f$ satisfies (\ref{Green-1end}), that is
	\begin{align}\label{f-eq1}
	\begin{split}
		\Delta _g f &=0,\\
		f&= \frac{1}{2} \text{ on } \Sigma,\\
		\lim_{x\to \infty_1} f(x) &=1.
	\end{split}
\end{align}
So by (\ref{cap-phi}) and (\ref{phi-expan}),
\begin{align}\label{cap-f}
c_1 = \frac{1}{2}\mathcal{C}(\Sigma, g) = \frac{1}{4 \pi } \int_{\Sigma}|\nabla f|.
\end{align}
Similarly, on $(M^{\prime}, g)$, $1-f$ also satisfies (\ref{Green-1end}), so 
\begin{align}
	c_2 = \frac{1}{2} \mathcal{C}(\Sigma, g) = c_1.
\end{align}

We now introduce an auxiliary function $u := t_0\cdot  \frac{1-f}{f}$ for some $t_0>0$, and the conformal metric $h = f^{4} g$. We always assume $m(g) >0$ in the following. Then $R_h = f^{-4}R_g \geq 0$.
Notice that $u \in (0, t_0]$ on $M$, $u = t_0$ on $\Sigma $ and $u(x) \to 0$ as $x \to \infty_1$. Moreover,
\begin{align*}
	\Delta _h u =0.
\end{align*}

Since $u$ is a proper smooth map on $M$, by Sard's theorem, the regular values of $u$ is an open dense subset of $(0, t_0]$. For any regular value $t \in (0, t_0)$, we define 
\begin{align*}
	M_t := \{ t \leq u \leq t_0\} ,\ \ \Sigma_t := \{ u =t\} .
\end{align*}
Since $\Sigma $ is connected by assumptions, using maximum principle, we know that a regular level set $\Sigma_t$ is also connected and separates $\Sigma$ from $\infty_1$. In particular, $\Sigma_t$ is a connected $2$-sphere.

\begin{prop}\label{f-prop}
	We have the following integration inequality for $u$ on $(M^3, h)$:
	\begin{align*}
		1 &- \frac{\mathcal{C}(\Sigma, g) ^2}{m(g)^2} \geq \\
				 &\ \ \frac{1}{8 \pi t_0 }\int_{(M, h)}\left( \frac{\left|\nabla ^2 u + u^{-1}|\nabla u|_h^2 (h - 3 \nu \otimes \nu )\right|_h^2}{|\nabla u|_h} + R_h |\nabla u|_h \right) \mathrm{dvol}_h,
	\end{align*}
	where $\nu = \frac{\nabla u}{|\nabla u|_h}$, and the integral is taken over the regular set of $u$.
\end{prop}
\begin{proof}
We first smooth $|\nabla u|_h$ by defining for any $\epsilon >0$,
$$
\phi _\epsilon := \sqrt{|\nabla u|_h^2 + \epsilon } .
$$

If $\Sigma_t$ is a regular level set of $u$, then the Gauss-Codazzi equation implies that
\begin{align*}
	R_h - 2 \Ric_h\left( \nu , \nu  \right) = R_{\Sigma_t} + |II|^2 - H^2,
\end{align*}
where $\nu = \frac{\nabla u}{|\nabla u|_h}$ and $ II = \frac{\nabla ^2_{\Sigma_t} u}{|\nabla u|_h}, H = \mathrm{tr}_{\Sigma_t} II $ are the second fundamental form and mean curvature of $\Sigma_t$ in $(M^3, h)$ respectively. 
So 
\begin{align}\label{GC}
	\begin{split}
	(R_h - R_{\Sigma_t}) |\nabla u|_h^2 &= 2 \Ric_h( \nabla u, \nabla u) + |\nabla ^2 u|_h^2 - 2 |\nabla |\nabla u|_h |_h^2\\
					  &\ \ \ - (\Delta_h u)^2 + 2 \Delta_h u \cdot \nabla_h ^2 u(\nu , \nu ) \\
	&= 2 \Ric_h( \nabla u, \nabla u) + |\nabla ^2 u|_h^2 - 2 |\nabla |\nabla u|_h |_h^2 ,
\end{split}
\end{align}
where we used the equation $\Delta_h u =0$. Together with the Bochner formula
\begin{align*}
	\Delta_h |\nabla u|_h^2 = 2 |\nabla ^2 u|_h^2 + 2 \left<\nabla \Delta_h u, \nabla u \right>_h +2 \Ric_h(\nabla u, \nabla u),
\end{align*}
we have
\begin{align*}
	\Delta_h |\nabla u|^2 = |\nabla ^2 u|_h^2 + 2|\nabla |\nabla u|_h |_h^2 + (R_h - R_{\Sigma_t}) |\nabla u|_h^2.
\end{align*}

At any point on a regular level set $\Sigma_t$,
\begin{align*}
	\begin{split}
		\Delta_h \phi _\epsilon &= \frac{1}{2} \phi _\epsilon ^{-1} \Delta |\nabla u|_h^2 - \frac{1}{4} \phi _\epsilon ^{-3} |\nabla |\nabla u|_h^2 |_h^2\\
				      &= \frac{1}{2} \phi _\epsilon ^{-1}\left( |\nabla ^2 u|_h^2 + (R_h - R_{\Sigma_t}) |\nabla u|_h^2 \right) + \frac{\epsilon }{\phi _\epsilon ^{3}} |\nabla |\nabla u|_h |_h^2 \\
				      &\geq \frac{1}{2} \phi _\epsilon ^{-1}\left( |\nabla ^2 u|_h^2 + (R_h - R_{\Sigma_t})|\nabla u|_h^2 \right) .
	\end{split}
\end{align*}
\begin{comment}
Notice that the mean curvature of $\Sigma_t$ is
\begin{align*}
	H_{\Sigma_t} = \frac{\Delta f - \nabla ^2 f(\nu , \nu )}{|\nabla f|} = - \frac{\nabla ^2 f (\nu , \nu )}{|\nabla f|}.
\end{align*}
\end{comment}

Taking integration on $M_t$ and using integration by parts and co-area formula, we have
\begin{align}\label{lap-ineq1}
	\begin{split}
		\int_{\Sigma } \langle \nabla \phi _\epsilon , \nu  \rangle_h-	\int_{\Sigma_t} &\left<\nabla \phi _\epsilon , \nu  \right>_h \\
	&\geq \frac{1}{2} \int_{M_t} \phi _\epsilon ^{-1}\left( |\nabla ^2 u|_h^2+ R_h|\nabla u|_h^2 \right) - \frac{1}{2} \int_{t}^{t_0} \int_{\Sigma_s} R_{\Sigma_s} \frac{|\nabla u|_h}{\phi _\epsilon } ds.
	\end{split}
\end{align}

Define the tensor 
\begin{align*}
	\mathcal{T}_1 := \nabla ^2 u + u ^{-1} |\nabla u|_h^2(h - 3 \nu \otimes \nu ).
\end{align*}
Then
\begin{align*}
	|\mathcal{T}_1|_h^2 = |\nabla ^2 u|^2_h + 6 u^{-2} |\nabla u|_h^{4} - 6 u ^{-1} |\nabla u|_h^2 \nabla ^2u (\nu ,\nu ).
\end{align*}
Notice that
\begin{align*}
	u^{-1} |\nabla u|_h \nabla ^2 u(\nu ,\nu ) &= u^{-1}|\nabla u|_h \langle \nabla |\nabla u|_h, \nu  \rangle_h\\
						   &= \mathrm{div}(|\nabla u|_h u^{-1} \nabla u) - |\nabla u| \mathrm{div}(u^{-1} \nabla u)\\
						   &= \mathrm{div} (|\nabla u|_h u^{-1}\nabla u) + u^{-2} |\nabla u|_h^{3}.
\end{align*}
So
\begin{align}\label{T-tensor1}
	|\mathcal{T}_1|_h^2 = |\nabla ^2 u|_h^2 - 6 |\nabla u|_h \mathrm{div} (|\nabla u|_h u^{-1}\nabla u).
\end{align}

Substituting (\ref{T-tensor1}) into (\ref{lap-ineq1}), and using the Gauss-Bonnet theorem, we have
\begin{align}\label{lap-ineq2'}
	\begin{split}
		\int_{\Sigma} &\left<\nabla \phi _\epsilon , \nu  \right>_h - \int_{\Sigma_t} \left<\nabla \phi _\epsilon , \nu  \right>_h\\
		 & \geq \frac{1}{2} \int_{M_t} \phi _\epsilon ^{-1}\left( |\mathcal{T}_1|_h^2 + R_h |\nabla u|_h^2 \right) + 3 \int_{M_t} \frac{|\nabla u|_h}{\phi _\epsilon } \mathrm{div}( |\nabla u|_h u^{-1} \nabla u)\\
		 &\ \ - \frac{1}{2} \int_{t}^{t_0} \int_{\Sigma_s} R_{\Sigma_s} \frac{|\nabla u|_h}{\phi _\epsilon } ds\\									   &\geq \frac{1}{2} \int_{M_t} \phi _\epsilon ^{-1}\left( |\mathcal{T}_1|_h^2 + R_h |\nabla u|_h^2 \right)   + 3 \int_{\Sigma } u^{-1}|\nabla u|_h^{2} - 3 \int_{\Sigma _t} u^{-1}|\nabla u|_h^{2}\\									   &\ \  -3 \epsilon \cdot \int_{M_t} \frac{\mathrm{div}(|\nabla u| u^{-1} \nabla u) }{\phi _\epsilon ^2 + \phi _\epsilon |\nabla u|_h} \\
									   &\ \ - 4 \pi (t_0-t) + \frac{\epsilon }{2} \cdot  \int_{t}^{t_0} \int_{\Sigma_s}\frac{ R_{\Sigma_s} }{\phi _\epsilon ^2 + \phi _\epsilon |\nabla u|_h} ds.
	\end{split}
\end{align}

Using the asymptotical estimates (\ref{f-eq3}), we know
\begin{align}\label{error1}
	\begin{split}
	\epsilon \cdot \int_{M_t} \frac{\mathrm{div}(|\nabla u| u^{-1} \nabla u) }{\phi _\epsilon ^2 + \phi _\epsilon |\nabla u|_h} &\leq \epsilon \cdot \int_{M_t} \frac{u^{-1}|\nabla u|_h \nabla ^2 u(\nu ,\nu )}{\phi _\epsilon ^2 + \phi _\epsilon |\nabla u|_h} \\
																    &\leq C(h) \epsilon \cdot t ^{-3},
\end{split}
\end{align}
and
\begin{align}\label{error2}
	\epsilon \cdot  \int_{t}^{t_0} \int_{\Sigma_s}\frac{ R_{\Sigma_s} }{\phi _\epsilon ^2 + \phi _\epsilon |\nabla u|_h} ds &\leq C(h) \epsilon \cdot t ^{-3}.
\end{align}

\begin{comment}

For any regular value $t $ of $u$, define
\begin{align*}
	v_\epsilon (t) := \int_{\Sigma _{t}} \phi _\epsilon |\nabla u|_h \mathrm{d A}_h.
\end{align*}
Using the co-area formula, integration by parts and $\Delta _h u =0$, we have
\begin{align*}
	\int_{t}^{t_0} \int_{\Sigma _s} \langle \nabla \phi _\epsilon , \nu  \rangle_h ds &= \int_{M_t} \langle \nabla \phi _\epsilon ,\nabla u \rangle_h\\
											&= \int_{\Sigma } \phi _\epsilon |\nabla u|_h - \int_{\Sigma _t} \phi _\epsilon |\nabla u|_h.
\end{align*}
So $v_\epsilon (t)$ can be defined for any $t \in [0, t_0)$ and it is a Lipschitz function with 
\begin{align*}
	v_\epsilon '(t) =  \int_{\Sigma _{t}} \langle \nabla \phi _\epsilon , \nu  \rangle_h \text{ for a.e. } t \in (0, t_0).
\end{align*}

From (\ref{lap-ineq2'}), together with the Gauss-Bonnet theorem and estimates (\ref{error1}), (\ref{error2}), we have
\begin{align}\label{ODE-comp1}
	\begin{split}
	\int_\Sigma \langle \nabla \phi _\epsilon ,\nu  \rangle_h - v_\epsilon '(t)& \geq \frac{1}{2} \int_{M_t} \phi _\epsilon ^{-1}\left( |\mathcal{T}_1|_h^2 + R_h |\nabla u|_h^2 \right)\\
	&\ \ - 4 \pi (t_0-t) + 3 t_0^{-1} \int_{\Sigma } |\nabla u|_h^2 - 3 t ^{-1} v_\epsilon (t) - C(h) \epsilon t ^{-3}.\\
\end{split}
\end{align}
\end{comment}

For any regular value $t \in (0, 1)$ of $u$, as in the proof of \cite[Theorem 1.2]{BKKS'22}, we can divide the integrals into two disjoint parts such that one is the integral over preimage of an open set containing the critical values of $u$. First letting $\epsilon \to 0$, and then choosing a sequence of regular values $t_i \to 0$, together with the asymptotical estimates, we have
\begin{align*}
	\int_{\Sigma } \langle \nabla |\nabla u|_h, \nu  \rangle_h + 4 \pi t_0 - 3 t_0^{-1} \int_{\Sigma }|\nabla u|_h^2 \geq \frac{1}{2} \int_{M} \frac{1}{|\nabla u|_h} \cdot \left( |\mathcal{T}_1|_h^2 + R_h |\nabla u|_h^2 \right) \mathrm{dvol}_h.
\end{align*}

Since the mean curvature of $\Sigma $ in $(M,g)$ is zero, $\nabla ^2_g f (\nu ,\nu ) =0$, which implies that
\begin{align*}
	\int_{\Sigma } \langle \nabla |\nabla u|_h, \nu \rangle_h &= \frac{2}{t_0} \int_{\Sigma } |\nabla u|_h^2.
\end{align*}
So we have 
\begin{align}\label{ineq-good1}
	4 \pi t_0 - t_0^{-1} \int_{\Sigma } |\nabla u|_h^2 \geq \frac{1}{2} \int_{M} \frac{1}{|\nabla u|_h} \cdot \left( |\mathcal{T}_1|_h^2 + R_h |\nabla u|_h^2 \right) \mathrm{dvol}_h.
\end{align}
In particular, since $R_h \geq 0$,
\begin{align}\label{L2-upper-boundary}
	\int_{\Sigma }|\nabla u|_h^2 \leq 4 \pi t_0^2.
\end{align}

On the other hand, 
\begin{align*}
	\left( \int_{\Sigma } |\nabla u|_h \right) ^2 \leq \int_{\Sigma } |\nabla u|_h^2 \cdot \mathrm{Area}_h(\Sigma ) = \frac{1}{16}\int_{\Sigma } |\nabla u|_h^2 \cdot \mathrm{Area}_g(\Sigma ),
\end{align*}
and
\begin{align*}
	\int_{\Sigma } |\nabla u|_h \mathrm{d A}_h = t_0 \int_{\Sigma } |\nabla f|_g \mathrm{d A}_g = 2 \pi t_0 \mathcal{C}(\Sigma ,g),
\end{align*}
so
\begin{align*}
	\int_{\Sigma } |\nabla u|_h^2 \geq \frac{64 \pi ^2 t_0^2 \mathcal{C}(\Sigma ,g)^2}{\mathrm{Area}_g(\Sigma )}.
\end{align*}
Together with (\ref{ineq-good1}), we have
\begin{align}\label{area-cap-ineq}
	4 \pi t_0 \left( 1- \frac{16 \pi \mathcal{C}(\Sigma , g)^2}{\mathrm{Area}_g (\Sigma )} \right) \geq \frac{1}{2} \int_{M} \frac{1}{|\nabla u|_h} \cdot \left( |\mathcal{T}_1|_h^2 + R_h |\nabla u|_h^2 \right) \mathrm{dvol}_h.
\end{align}

By the Penrose inequality $\mathrm{Area}_g(\Sigma ) \leq 16 \pi m(g)^2$, we have
\begin{align*}
	4 \pi t_0 \left( 1- \frac{\mathcal{C}(\Sigma ,g)^2}{m(g)^2} \right) \geq \frac{1}{2} \int_{M} \frac{1}{|\nabla u|_h} \cdot \left( |\mathcal{T}_1|_h^2 + R_h |\nabla u|_h^2 \right) \mathrm{dvol}_h.
\end{align*}
This concludes the proof. 

\end{proof}

Notice that (\ref{area-cap-ineq}) gives another proof of the following mass-area-capacity inequality (cf. \cite[Theorem 1]{BM08}).

\begin{thm}\label{mass-area-capacity}
Let $(M^3,g)$ be a complete, one-ended asymptotically flat $3$-manifold with
nonnegative scalar curvature and a compact connected outermost minimal boundary $\Sigma $. Then
	\begin{align}\label{m-a-c}
		 \mathcal{C}(\Sigma,g) ^2 \leq \frac{1}{16 \pi }\Area_g(\Sigma) \leq m(g)^2.
	\end{align}
\end{thm}

By the symmetry of $(\tilde{M}, g)$, applying Proposition \ref{f-prop} to $\tilde{u}:= \frac{t_0^2}{u}$ and $\tilde{h}: = (1-f)^{4} g = \left( \frac{u}{t_0} \right) ^{4} h$, we have
\begin{align}\label{M'-ineq}
	1- \frac{\mathcal{C}(\Sigma ,g)^2}{m(g)^2} \geq \frac{1}{8 \pi t_0} \int_{(M', h)}\left(  \frac{\left| \nabla ^2 \tilde{u} - \tilde{u} ^{-1} |\nabla \tilde{u}|_h^2 (h - \nu \otimes \nu ) \right|_h^2}{ |\nabla \tilde{u}|_h} + R_h |\nabla \tilde{u}|_h \right)  \mathrm{dvol}_h.
\end{align}

Proposition \ref{f-prop} is not sufficient for our control of $f$ around infinity. For this purpose, using a similar argument together with (\ref{L2-upper-boundary}), we prove the following integration inequality.
\begin{prop}\label{u-good}
	For $u = \frac{2}{\mathcal{C}(\Sigma ,g)} \cdot \frac{1-f}{f}$, we have
\begin{align}
	1 - \frac{\mathcal{C}(\Sigma ,g)}{m(g)} \geq \frac{\mathcal{C}(\Sigma ,g)}{96 \pi} \int_{(M,h)} \frac{1}{|\nabla u|_h}\left( \left|\nabla ^2 u + |\nabla u|_h^{\frac{3}{2}} (h - 3 \nu \otimes \nu )\right|_h^2 + R_h |\nabla u|_h^2 \right) \mathrm{dvol}_h.
\end{align}	
\end{prop}

\begin{proof}
	The proof is almost the same as the proof of Proposition \ref{f-prop}. The difference is that instead of considering the tensor $\mathcal{T}_1$, we define 
\begin{comment}
Also from (\ref{ODE-comp1}), we know
\begin{align*}
	(t ^{-3} v_\epsilon (t) )' \leq t ^{-3}\left( 4 \pi (t_0-t) + C(h) \epsilon t ^{-3} + \int_{\Sigma } \langle \nabla \phi _\epsilon ,\nu  \rangle_h - 3 t_0^{-1} \int_{\Sigma }|\nabla u|_h^2 \right) .
\end{align*}
Integrating from $t = t_0$, for any regular $t \in (0, t_0)$, we have
\begin{align*}
	t ^{-3}	v_\epsilon (t) & \geq t_0^{-3} v_\epsilon (t_0) + \frac{1}{2}\left( 4 \pi t_0 + \int_{\Sigma }\langle \nabla \phi _\epsilon , \nu  \rangle_h - 3 t_0 ^{-1} \int_{\Sigma } |\nabla u|_h^2 \right) \left( t_0^{-2} - t ^{-2} \right)\\
	&\ \ - 4 \pi (t_0^{-1} - t ^{-1}) - C(h) \epsilon (t_0^{-5} - t ^{-5}),
\end{align*}
which implies that
\begin{align*}
	t ^{-1} v_\epsilon (t) & \geq t ^2t_0^{-3} v_\epsilon (t_0) + \frac{1}{2}\left( 4 \pi t_0 + \int_{\Sigma }\langle \nabla \phi _\epsilon , \nu  \rangle_h - 3 t_0 ^{-1} \int_{\Sigma } |\nabla u|_h^2 \right) \left( t ^2 t_0^{-2} - 1 \right)\\
	&\ \ - 4 \pi (t ^2 t_0^{-1} - t ) - C(h) \epsilon (t ^2 t_0^{-5} - t ^{-3}).
\end{align*}
First taking $\epsilon \to 0$ and then taking $t\to 0$, we have
\begin{align*}
	4 \pi t_0 + \int_{\Sigma }\langle \nabla |\nabla u|_h , \nu  \rangle_h - 3 t_0 ^{-1} \int_{\Sigma } |\nabla u|_h^2 \geq 0.
\end{align*}

Define the tensor
\end{comment}
\begin{align*}
	\mathcal{T}_2:= \nabla ^2 u + |\nabla u|_h^{\frac{3}{2}} (h - 3 \nu \otimes \nu ).
\end{align*}
Then
\begin{align}\label{T-tensor2}
	|\mathcal{T}_2|^2_h = |\nabla ^2u|^2_h + 6 |\nabla u|^{3}_h - 6 |\nabla u|^{\frac{3}{2}}_h \nabla ^2 u (\nu ,\nu ).
\end{align}
Notice that
\begin{align*}
	|\nabla u|^{\frac{1}{2}}_h \nabla ^2 u (\nu ,\nu ) &= |\nabla u|_h^{\frac{1}{2}} \langle \nabla |\nabla u|_h, \nu  \rangle_h\\
							   &= \mathrm{div} \left( |\nabla u|_h^{\frac{1}{2}} \nabla u \right) - |\nabla u|_h \langle \nabla |\nabla u|^{-\frac{1}{2}}_h, \nabla u \rangle_h\\
							   &= \mathrm{div} \left( |\nabla u|_h^{\frac{1}{2}} \nabla u \right) + \frac{1}{2} |\nabla u|_h^{\frac{1}{2}} \langle \nabla |\nabla u|_h, \nu  \rangle_h,
\end{align*}
which implies that
\begin{align}\label{div-id1}
	|\nabla u|^{\frac{1}{2}}_h \nabla ^2 u (\nu ,\nu ) = 2 \mathrm{div} \left( |\nabla u|_h^{\frac{1}{2}} \nabla u \right) .
\end{align}
Substituting (\ref{T-tensor2}) and (\ref{div-id1}) into (\ref{lap-ineq1}), and using the Gauss-Bonnet theorem and asymptotical estimates, we have
\begin{comment}
\begin{align}\label{lap-id2}
	\begin{split}
	\frac{1}{2} \phi _\epsilon ^{-1} |\nabla^2 f|^2 &= \frac{1}{2} \phi _\epsilon ^{-1} |\mathcal{T}|^2  -3 \mathrm{div}\left( |\nabla f|^2 \phi _\epsilon ^{-1} f^{-1}(1-f)^{-1}(2f-1) \nabla f  \right)\\
							&\ \ + 6 \phi _\epsilon ^{-1} f^{-1}(1-f)^{-1}|\nabla f|^{4}\\
							&\ \ + \frac{3 \epsilon }{\phi _\epsilon ^2} \phi _\epsilon ^{-1}f^{-1}(1-f)^{-1}(2f-1) \nabla ^2 f(\nabla f, \nabla f).
\end{split}
\end{align}
So (\ref{lap-ineq1}) is equivalent to
\end{comment}
\begin{comment}
		 & \geq \frac{1}{2} \int_{M_t} \phi _\epsilon ^{-1}\left( |\mathcal{T}|_h^2 + R_h |\nabla u|_h^2 \right) - \frac{1}{2} \int_{t}^{t_0} \int_{\Sigma_s} R_{\Sigma_s} ds \\
									   &\ \ -3 \int_{M_t} |\nabla u|_h^{2} + 3 \int_{M_t}  |\nabla u|_h^{\frac{1}{2}} \nabla ^2 u (\nu ,\nu )\\
									   &\ \ +3 \epsilon \cdot \int_{M_t} \frac{ |\nabla u|_h^{\frac{1}{2}}\nabla ^2u(\nu ,\nu ) - |\nabla u|^2_h  }{\phi _\epsilon ^2 + \phi _\epsilon |\nabla u|_h} \\
									   &\ \ + \frac{\epsilon }{2} \cdot  \int_{t}^{t_0} \int_{\Sigma_s}\frac{ R_{\Sigma_s} }{\phi _\epsilon ^2 + \phi _\epsilon |\nabla u|_h} ds\\
		\end{comment}

\begin{align}\label{lap-ineq2}
	\begin{split}
		\int_{\Sigma} &\left<\nabla \phi _\epsilon , \nu  \right>_h - \int_{\Sigma_t} \left<\nabla \phi _\epsilon , \nu  \right>_h\\									   &\geq \frac{1}{2} \int_{M_t} \phi _\epsilon ^{-1}\left( |\mathcal{T}_2|_h^2 + R_h |\nabla u|_h^2 \right) - 4 \pi (t_0 -t) \\
									   &\ \ -3 \int_{M_t} |\nabla u|_h^{2} + 6 \int_{\Sigma } |\nabla u|_h^{\frac{3}{2}} - 6 \int_{\Sigma _t} |\nabla u|_h^{\frac{3}{2}}- C(g, t) \cdot \epsilon.
	\end{split}
\end{align}
\begin{comment}
where in the last inequality, we used the asymptotical estimates (\ref{f-eq3}) and the fact that $h$ is asymptotically flat.

For any regular value $t \in (0, 1)$ of $u$, as in the proof of \cite[Theorem 1.2]{BKKS'22}, we can divide the integrals into two disjoint parts such that one is the integral over preimage of an open set containing the critical values of $u$, then letting $\epsilon \to 0$, and using Gauss-Bonnet theorem, we have
\begin{align*}
	\int_{\Sigma} &\left<\nabla |\nabla u|_h , \nu  \right>_h - \int_{\Sigma_t} \left<\nabla |\nabla u|_h , \nu  \right>_h\\
		 & \geq \frac{1}{2} \int_{M_t} \frac{1}{|\nabla u|_h}\left( |\mathcal{T}|_h^2 + R_h |\nabla u|_h^2 \right) - 4 \pi (t_0- t) \\
									   &\ \ -3 \int_{M_t} |\nabla u|_h^{2} + 3 \int_{\Sigma } |\nabla u|_h^{\frac{3}{2}} - 3 \int_{\Sigma _t} |\nabla u|_h^{\frac{3}{2}}.
\end{align*}
\end{comment}
First taking $\epsilon \to 0$, and then taking $t\to 0$, we have
\begin{align*}
	\int_{\Sigma} \left<\nabla |\nabla u|_h , \nu  \right>_h + 4 \pi t_0 + 3 \int_M |\nabla u|_h^2 & -  6 \int_{\Sigma } |\nabla u|_h^{\frac{3}{2}}\\
	&\geq \frac{1}{2} \int_{M} \frac{1}{|\nabla u|_h}\left( |\mathcal{T}_2|_h^2 + R_h |\nabla u|_h^2 \right) \mathrm{dvol}_h.
\end{align*}

Notice that
\begin{align*}
	\int_{\Sigma _t}|\nabla u|_h &= \int_{\Sigma } |\nabla u|_h = 2 \pi t_0 \mathcal{C}(\Sigma ,g),
				     \end{align*}
so
\begin{align*}
	\int_{M} |\nabla u|_h^2 &= \int_0^{t_0} \int_{\Sigma _t} |\nabla u|_h 
				= 2 \pi t_0^2 \mathcal{C}(\Sigma ,g).
\end{align*}
Using (\ref{L2-upper-boundary}), we have
\begin{align*}
	12 \pi t_0 + 6 \pi t_0^2 \mathcal{C}(\Sigma ,g) - 6 \int_{\Sigma } |\nabla u|_h^{\frac{3}{2}} \geq \frac{1}{2} \int_{M} \frac{1}{|\nabla u|_h}\left( |\mathcal{T}_2|_h^2 + R_h |\nabla u|_h^2 \right) \mathrm{dvol}_h.
\end{align*}

By H\"older inequality, 
\begin{align*}
	\int_{\Sigma }|\nabla u|_h \leq \left( \int_{\Sigma }|\nabla u|_h^{\frac{3}{2}} \right) ^{\frac{2}{3}} \cdot \mathrm{Area}_h(\Sigma ) ^{\frac{1}{3}},
\end{align*}
we know
\begin{align*}
	\int_{\Sigma }|\nabla u|_h^{\frac{3}{2}} \geq \frac{4 \left( 2 \pi t_0 \mathcal{C}(\Sigma ,g) \right) ^{\frac{3}{2}}}{\mathrm{Area}_g (\Sigma ) ^{\frac{1}{2}}}.
\end{align*}
Thus
\begin{align*}
	1 + \frac{t_0 \mathcal{C}(\Sigma ,g)}{2} - 2\left( \frac{8 \pi t_0 \mathcal{C}(\Sigma ,g)^{3}}{\mathrm{Area}_g(\Sigma )} \right) ^{\frac{1}{2}} \geq \frac{1}{24 \pi t_0} \int_{M} \frac{1}{|\nabla u|_h}\left( |\mathcal{T}_2|_h^2 + R_h |\nabla u|_h^2 \right) \mathrm{dvol}_h.
\end{align*}
If we take $t_0 = \frac{2}{\mathcal{C}(\Sigma ,g)}$, then 
\begin{align*}
	1 - \left( \frac{16 \pi \mathcal{C}(\Sigma ,g)^2}{\mathrm{Area}_g(\Sigma )} \right) ^{\frac{1}{2}} \geq \frac{\mathcal{C}(\Sigma ,g)}{96 \pi} \int_{M} \frac{1}{|\nabla u|_h}\left( |\mathcal{T}_2|_h^2 + R_h |\nabla u|_h^2 \right) \mathrm{dvol}_h.
\end{align*}
Together with the Penrose inequality, this concludes the proof.
\end{proof}

Since at any regular point,
\begin{align*}
	\left| \nabla \left( |\nabla u|_h^{\frac{1}{2}} - u\right) \right|_h &\leq \frac{\left| \nabla |\nabla u|_h - 2 |\nabla u|_h^{\frac{3}{2}} \nu \right|_h}{2 |\nabla u|_h^{\frac{1}{2}}}\\
						   &\leq \frac{ \left|\nabla ^2 u + |\nabla u|_h^{\frac{3}{2}} (h - 3 \nu \otimes \nu )\right|_h}{2 |\nabla u|_h^{\frac{1}{2}}},
\end{align*}
we have the following estimate
\begin{align}
	\int_{(M,h)} \left|\nabla \left( |\nabla u|_h^{\frac{1}{2}} - u\right) \right|_h^2 \mathrm{dvol}_h \leq C\cdot \left( \frac{1}{\mathcal{C}(\Sigma ,g)}- \frac{1}{m(g)} \right) .
\end{align}

Similarly, for $\tilde{u} = \frac{4}{\mathcal{C}(\Sigma ,g)^2 u}$
 on $(M', h)$, we have
\begin{align*}
	\left|\nabla |\nabla \tilde{u}|^{\frac{1}{2}}_h \right|_h & = \frac{|\nabla |\nabla \tilde{u}|_h|_h}{2|\nabla \tilde{u}|_h^{\frac{1}{2}}} 
						       \leq \frac{\left| \nabla ^2 \tilde{u} - \tilde{u}^{-1}|\nabla \tilde{u}|_h^2(h - \nu \otimes \nu ) \right|_h }{2 |\nabla \tilde{u}|_h^{\frac{1}{2}}},
\end{align*}
which together with (\ref{M'-ineq}) implies that 
\begin{align}
	\int_{(M', h)} \left|\nabla \left( |\nabla \tilde{u}|_h^{\frac{1}{2}} - \frac{2}{\mathcal{C}(\Sigma ,g)}\right) \right|_h^2 \mathrm{dvol}_h \leq \frac{C}{\mathcal{C}(\Sigma ,g)}\cdot \left( 1- \frac{\mathcal{C}(\Sigma ,g)^2}{m(g)^2} \right) .
\end{align}
Notice that $|\nabla u|_h^{\frac{1}{2}} - u$ and $|\nabla \tilde{u}|_h^{\frac{1}{2}} - \frac{2}{\mathcal{C}(\Sigma ,g)}$ agree on the middle sphere $\Sigma $. Define 
\begin{align}\label{P-defn}
	\begin{split}
		P(x) &:= \left| |\nabla u|_h^{\frac{1}{2}}(x) - u(x) \right| \text{ when } x \in M;\\
		P(x)&:=\left| |\nabla \tilde{u}|_h^{\frac{1}{2}}(x) - \frac{2}{\mathcal{C}(\Sigma ,g)} \right| \text{ when } x \in M'.
\end{split}
\end{align}
Then $P$ is a continuous function, smooth around the end $\infty_1$ in $M$, and we have proved the following $L^2$-estimate for its gradient.
\begin{prop}\label{P-prop}
	Let $(\tilde{M}, g)$ be a two-ended asymptotically flat $3$-manifolds obtained as in Proposition \ref{doubling-approx}, and assume $m(g) \geq m_0>0$. Let $f$ be a solution of (\ref{Green-2end}), and $h := f ^{4} g$ a conformal metric on $\tilde{M}$. For the function $P$ defined in (\ref{P-defn}), there exists a uniform constant $C>0$ such that
\begin{align}\label{nbl-norm-f}
\int_{(\tilde{M}, h)} |\nabla P|^2_h \mathrm{dvol}_h \leq \frac{C}{\mathcal{C}(\Sigma ,g)} \left( 1- \frac{\mathcal{C}(\Sigma ,g)^2}{m(g)^2} \right) .	
\end{align}
Moreover, 
\begin{align}\label{nbl-f-limit}
	\lim_{x\to \infty_1} P(x) = \lim_{x\to \infty_2} P(x) = 0.
\end{align}
\end{prop}

\section{Harmonic coordinate for conformal metric}
Let $(\tilde{M}, g)$ be a two-ended asymptotically flat $3$-manifolds obtained as in Proposition \ref{doubling-approx},  $f$ the harmonic function defined by (\ref{har-eq}) on $(\tilde{M}^{3}, g)$, and $h:= f^{4}g$ the conformal metric. Then on the end $\infty_2$, since $g$ is harmonically flat, we have $h_{ij}(x) = f^{4}(x) V^{4}(x) \delta_{ij} $, where $V(x)$ is a positive bounded $\delta $-harmonic function. So on the end $\infty_2$, $h$ is conformal to a punctured ball with the conformal factor $(fV)^{4}(x)$, where $(fV)(x)$ is a bounded $\delta $-harmonic function in the punctured ball. Hence, by the removable singularity theorem, $fV$ can be extended to the whole ball, which together with the expansion (\ref{f-eq2}) and (\ref{V-expan}) implies that $h$ can be extended smoothly over the one point compactification $\tilde{M}^*:=\tilde{M}\cup \{\infty_2\}$. Also the function $P$ defined in (\ref{P-defn}) can be extended continuously to $\tilde{M}^*$. 

By standard computations, $(\tilde{M}^*, h)$ also has nonnegative scalar curvature and has a single harmonically flat end $\infty_1$ with ADM mass $m(h)= m(g) -  \mathcal{C}(\Sigma, g) \geq 0$ (cf. \cite[Equation (84)]{Bray'01}).

In the following, we assume that $|m(g) - m_0 | \ll 1$ and $|\mathcal{C}(\Sigma ,g) - m_0| \ll 1$ for a fixed $m_0>0$. By the mass-area-capacity inequality, $|\mathrm{Area}_g(\Sigma ) - 16 \pi m_0^2| \ll 1$.

\begin{comment}

For $g$-harmonic function $f$, since $|\nabla f |_g^2 = f^{4}|\nabla f |_h^2$ and $\mathrm{dvol}_g = f^{-6} \mathrm{dvol}_h$, (\ref{nbl-norm-f}) is equivalent to 
\begin{align}\label{nbl-nbl-f-h}
	\int_{\tilde{M}^*} \frac{|\nabla \left( (1-f)^{-4} |\nabla f|_h^2 \right) |_h^2}{(1-f)^{-8} |\nabla f|_h^3} \mathrm{dvol}_h \leq C(m_0)\cdot m(h). 
\end{align}
Notice that we also have
\begin{align}\label{nbl-L^2-identity}
	\int_{\tilde{M}^*} f^{-2}|\nabla f|_h^2 \mathrm{dvol}_h = 2\pi \mathcal{C}(\Sigma ,g).
\end{align}

\end{comment}

Let $\{x^{j} \}_{j=1}^{3}$ denote the asymptotically flat coordinate system of the end $\infty_1$. We firstly solve the harmonic coordinate functions $u^j$, for each $j \in \{1,2,3\}$, such that 
\begin{align}
	\begin{split}
		\Delta _{h} u^j &=0,\\
		|u^j(x)- x^j| &= o(|x|^{1-\sigma }) \text{ as } x \to \infty_1,
	\end{split}
\end{align}
where $\sigma > \frac{1}{2}$ is the order of the asymptotic flatness. 
Denote by $\mathbf{u}$ the resulting harmonic map
\begin{align*}
	\mathbf{u}:= (u^1, u^2, u^3): (\tilde{M}^*, h) \to \mathbb{R}^3.
\end{align*}

For any fixed small $0< \epsilon \ll 1$, by \cite{DS'23}, we know that there exists a connected region $\mathcal{E}_1  \subset (\tilde{M}^*, h)$ containing $\infty_1$, with smooth boundary, such that 
$$
\mathrm{Area}_h(\partial \mathcal{E}_1) \leq m(h)^{2-\epsilon },
$$ and $\mathbf{u}: \mathcal{E}_1 \to \mathcal{Y}_1:= \mathbf{u}(\mathcal{E}_1) \subset \mathbb{R}^3$ is a diffeomorphism with the Jacobian satisfying 
$$
|\mathrm{Jac}\,\mathbf{u} - \mathrm{Id}| \leq \Psi (m(h) ),
$$
and under the identification by $\mathbf{u}$, the metric tensor satisfies
$$
\sum_{j,k=1}^{3}(h_{jk} - \delta_{jk})^2 \leq m(h)^{2 \epsilon }.
$$
Moreover, $d_{pGH}\left( (\mathcal{E}_1, \hat{d}_{\mathcal{E}_1}), (\mathbb{R}^3, d_{\mathrm{Eucl}}) \right) \leq m(h)^{\epsilon }$.

Applying \cite[Lemma 3.2, Lemma 3.3]{DS'23} to the function $P$ defined in (\ref{P-defn}), and using a uniform approximation by smooth functions (cf. \cite[Theorem 4.3]{EG15}) and Proposition \ref{P-prop}, we can find a connected closed region $\mathcal{E}_2 \subset \mathcal{E}_1 \cap \{P \leq m(h)^{\epsilon }\} $, which contains a neighborhood of the infinity, and $\partial \mathcal{E}_2$ is a compact smooth surface satisfying
\begin{align}
	\mathrm{Area}_h (\partial \mathcal{E}_2) \leq C(m_0)\cdot m(h)^{1- \epsilon }.
\end{align}

We claim that $\mathcal{E}_2 \cap \Sigma \neq \emptyset$. Otherwise, since a neighborhood of the infinity is included in $\mathcal{E}_2$, a connected component of $\partial \mathcal{E}_2$ will enclose $\Sigma $ in $M$. By the outermost minimal property of $\Sigma $, we know $\mathrm{Area}_g (\Sigma ) \leq \mathrm{Area}_g (\partial \mathcal{E}_2) \leq 16 \mathrm{Area}_h(\partial \mathcal{E}_2) \leq C(m_0)\cdot m(h)^{1-\epsilon }$, which contradicts with the fact that $|\mathrm{Area}_g (\Sigma ) - m_0| \ll 1$ and $m(h) \ll 1$.

Notice that for any $x$ such that $P(x) \leq m(h)^{\epsilon }$, from (\ref{P-defn}), if $f(x) \leq \frac{1}{2}$, we have 
\begin{align*}
	\left| (1-f)^{-1} |\nabla f|_h^{\frac{1}{2}} - \left( \frac{2}{\mathcal{C}(\Sigma ,g)} \right) ^{\frac{1}{2}} \right| \leq C(m_0)\cdot m(h)^{\epsilon };
\end{align*}
if $f(x) \geq \frac{1}{2}$, we have
\begin{align*}
	f^{-1}(1-f) \left| (1-f)^{-1} |\nabla f|_h^{\frac{1}{2}} - \left( \frac{2}{\mathcal{C}(\Sigma ,g)} \right) ^{\frac{1}{2}} \right| \leq C(m_0)\cdot m(h)^{\epsilon }.
\end{align*}

\begin{comment}
Fix any base point $p \in \Sigma \cap \mathcal{E}_2$.
We claim that for any $D>0$,
\begin{align*}
	\hat{B}_{h, \mathcal{E}_2}(p, D) \subset \{f< 1- m(h)^{\epsilon }\} .
\end{align*}
To see this, for any $x \in \hat{B}_{h, \mathcal{E}_2}(p, D)$, choose a geodesic $\gamma $ from $p$ to $x$, and without loss of generality we can assume $\gamma \subset \{ f \geq \frac{1}{2}\} $. Then $|\nabla f|_h^{\frac{1}{2}}(\gamma (t) ) \leq C(m_0)$, which implies that 
\begin{align*}
	f(x) \leq 
\end{align*}
such that $f(x) = 1- m(h)^{\epsilon }$, then there exists a curve $\gamma \subset \mathcal{E}_2$ such that $\gamma (0)=p$, $\gamma (1) = x$, and $\mathrm{Length}_h(\gamma ) \leq 10 m(h)^{-\frac{\epsilon }{2}}$. But by definition of $\mathcal{E}_2$, we have
\begin{align*}
	(1-f(x) )^{-1} &\leq (1- f(p) )^{-1} + \int_{\gamma } (1-f)^{-2} |\nabla f|_h \\
	&\leq 2 + C \cdot m(h)^{-\frac{\epsilon }{2}},
\end{align*}
which is a contradiction. 

For all $0< m(h) \ll 1$, we can apply the same arguments of \cite[Section 4]{DS'23} to $\hat{B}_{h, \mathcal{E}_2}(p, 10 m(h)^{- \frac{\epsilon }{2}})$, and as a result, we can find a smooth domain $\mathcal{E}$ such that $\mathrm{Area}( \partial \mathcal{E}) \leq C m(h)^{\frac{1}{2} - 4 \epsilon }$, the induced length metric of $h$ on $\mathcal{E}$ is almost the same as Euclidean metric, and $\hat{B}_{h, \mathcal{E}}(p, m(h)^{- \frac{\epsilon }{2}}) \subset \mathcal{E}_2$. 

\end{comment}
In summary, we have the following proposition.

\begin{prop}\label{DS-har}
	There exists a connected region $\mathcal{E}  \subset (\tilde{M}^*, h)$ containing $\infty_1$, with smooth boundary, such that 
$$
\mathrm{Area}_h(\partial \mathcal{E}) \leq \Psi (m(h) ),
$$ and $\mathbf{u}: \mathcal{E} \to \mathcal{Y}:= \mathbf{u}(\mathcal{E}) \subset \mathbb{R}^3$ is a diffeomorphism with the Jacobian satisfying 
$$
|\mathrm{Jac}\,\mathbf{u} - \mathrm{Id}| \leq \Psi (m(h) ),
$$
and under the identification by $\mathbf{u}$, the metric tensor satisfies
$$
\sum_{j,k=1}^{3}(h_{jk} - \delta_{jk})^2 \leq \Psi (m(h) ).
$$ 
For any base point $p \in \mathcal{E} \cap \Sigma $, any $x_0 \in \mathbb{R}^3$ and any fixed $D>0$, 
\begin{align*}
	d_{pGH}\left( (\hat{B}_{h, \mathcal{E}}(p, D), \hat{d}_{h, \mathcal{E}}, p), (B_{\mathrm{Eucl}}(x_o, D), d_{\mathrm{Eucl}}, x_o) \right) \leq \Psi (m(h)|D),
\end{align*}
and $\Phi _{\mathbf{u}(p)} \circ \mathbf{u}$ gives a $\Psi (m(h)|D)$-pGH approximation, where $\Phi _{\mathbf{u}(p)}$ is the translation diffeomorphism of $\mathbb{R}^3$ mapping $\mathbf{u}(p)$ to $x_o$. 

Moreover, inside $ \mathcal{E}$, we have
\begin{align}\label{grad-f}
\left| |\nabla f|_h -  \frac{2(1-f)^2}{\mathcal{C}(\Sigma ,g)} \right| \leq \Psi (m(h) ).	
\end{align}

\end{prop}

\section{$W^{1,2}$-convergence of elliptic equations}\label{conv-equations}
Assume that $(\tilde{M}_i^{3}, g_i)$ is a sequence of complete two-ended asymptotically flat $3$-manifolds obtained as in Proposition \ref{doubling-approx} with $m(g_i)\to m_0$ and $  \mathcal{C}(\Sigma_i, g_i)\to m_0$ for some positive constant $m_0>0$, where $\Sigma_i \subset (\tilde{M}_i^3, g_i)$ is the minimal surface such that $(\tilde{M}_i^3, g_i)$ is symmetric about $\Sigma_i$. Let $\varepsilon _i := m(g_i) - \mathcal{C}(\Sigma _i ,g_i) \to 0$. Notice that from (\ref{m-a-c}), $|\Area_{g_i}(\Sigma_i) - 16 \pi m_0^2| \to 0$, and particularly $ \Area_{g_i}(\Sigma_i) \geq A_0 >0$ for some uniform $A_0>0$.

Let $f_i$ be the harmonic functions defined by (\ref{har-eq}) on $(\tilde{M}^{3}_i, g_i)$, $(M_i, g_i)$ be the half of $(\tilde{M}_i^{3}, g_i)$ such that $f_i$ satisfies (\ref{f-eq1}), and $h_i:= f_i^{4}g_i$ the conformal metrics. Using the same notations as in previous section, let $\tilde{M}^*_i$ be the one point compactification $\tilde{M}_i \cup \{\infty_2\} $, then $(\tilde{M}^*_i, h_i)$ is a sequence of one-ended asymptotically flat $3$-manifolds with nonnegative scalar curvature and ADM mass $m(h_i)= \varepsilon _i \to 0$. 

Let $\mathcal{E}_i$ be regions given by Proposition \ref{DS-har}.  Taking any base point $p_i \in \mathcal{E}_i \cap \Sigma_i = \mathcal{E}_i \cap \{ f_i = \frac{1}{2}\} $, then
\begin{align*}
	(\mathcal{E}_i, \hat{d}_{h_i, \mathcal{E}_i}, p_i) \to (\mathbb{R}^3, d_{\mathrm{Eucl}}, x_o)
\end{align*}
in the pointed measured Gromov-Hausdorff topology, where $x_o \in \mathbb{R}^3$, the harmonic maps $\mathbf{u}_i$ with $\mathbf{u}_i(p_i)=x_o$ are $\Psi (\varepsilon _i)$-pGH approximation, and for any $D>0$, $(\mathbf{u}_i)_{\sharp} (\mathrm{dvol}_{h_i}|_{\hat{B}_{h_i, \mathcal{E}_i}(p_i, D)})$ weakly converges to $\mathrm{dvol}_{\mathrm{Eucl}}|_{B(x_o,D)}$ as $i\to \infty$.

By (\ref{grad-f}) and the Arzel\`{a}-Ascoli theorem, up to a subsequence, $f_i \to f_\infty$ locally uniformly for some nonnegative bounded Lipschitz function $f_\infty$ on $\mathbb{R}^3$. 

\begin{lem}\label{weak-W12}
	Up to a subsequence, $f_i$ converges to $f_\infty$ in the weakly $W^{1,2}$-sense. That is, for any uniformly converging sequence of compactly supported Lipschitz functions $\psi _i \to \psi $ and $\nabla \psi _i \to \nabla \psi $ in $L^2$, we have 
	\begin{align*}
		\lim_{i\to \infty} \int_{\mathcal{E}_i} f _i \psi _i \mathrm{dvol}_{h_i} & = \int_{\mathbb{R}^3} f _\infty \psi \mathrm{dvol}_{\mathrm{Eucl}},\\
		\lim_{i\to \infty} \int_{\mathcal{E}_i} \left<\nabla f _i, \nabla \psi _i \right>_{h_i} \mathrm{dvol}_{h_i} &= \int_{\mathbb{R}^3} \left<\nabla f _\infty, \nabla \psi  \right>_{\mathrm{Eucl} } \mathrm{dvol}_{\mathrm{Eucl}}.
	\end{align*}
\end{lem}
\begin{proof}
	Under the diffeomorphism $\mathbf{u}_i$, we can identify $\mathcal{E}_i$ as a subset in $\mathbb{R}^3$. Suppose that $\psi _i, \psi $ have support in $U \subset B(x_o,D)$ for some $D>0$. Then since $f _i \to f _\infty$ uniformly, by Proposition \ref{DS-har} we know
	\begin{align*}
		\lim_{i\to \infty} & \left| \int_{U \cap \mathcal{E}_i} f _i \psi _i \mathrm{dvol}_{h_i} - \int_{U} f _\infty \psi \mathrm{dvol}_{\mathrm{Eucl}} \right| \\
		& \leq \lim_{i\to \infty}  \int_{U \cap \mathcal{E}_i} |f _i \psi _i - f _\infty \psi | \mathrm{dvol}_{h_i} + \lim_{i\to \infty} \left| \int_{U\cap \mathcal{E}_i} f _\infty \psi \mathrm{dvol}_{h_i} - \int_{U} f _\infty \psi \mathrm{dvol}_{\mathrm{Eucl}} \right| \\
		& \leq C \lim_{i\to \infty} \int_{U \cap \mathcal{E}_i}|f _i \psi _i - f _\infty \psi | \mathrm{dvol}_{\mathrm{Eucl}}\\
		&=0.
	\end{align*}
	Similarly, since $|\nabla f _i|_{h_i}$ and $|\nabla \psi _i|_{h_i}$ are uniformly bounded and $h_i$ converges uniformly to $g_{\mathrm{Eucl}} $,
	\begin{align*}
		\lim_{i\to \infty} & \left| \int_{U \cap \mathcal{E}_i} \left< \nabla f _i, \nabla  \psi _i\right>_{h_i} \mathrm{dvol}_{h_i} - \int_{U} \left< \nabla f _\infty, \nabla  \psi\right>_{\mathrm{Eucl} } \mathrm{dvol}_{\mathrm{Eucl}} \right| \\
		& \leq C \lim_{i\to \infty} \int_{U \cap \mathcal{E}_i}\left| \left<\nabla f_i, \nabla \psi _i \right>_{h_i} - \left<\nabla f _\infty, \nabla \psi  \right>_{\mathrm{Eucl} } \right| \mathrm{dvol}_{\mathrm{Eucl}} \\
		&= C \lim_{i\to \infty} \int_{U \cap \mathcal{E}_i} \left|h_i^{jk} \partial _j f _i \partial _k \psi_i - \delta ^{jk} \partial _j f _\infty \partial _k \psi \right| \mathrm{dvol}_{\mathrm{Eucl}}\\
		&\leq C \lim_{i\to \infty} \sum_{j=1}^{3} \int_{U \cap \mathcal{E}_i}\left( |\left<\partial _j f _i - \partial _j f _\infty \right>|\cdot |\partial _j \psi | + | \partial _j \psi _i - \partial _j \psi |\right) \mathrm{dvol}_{\mathrm{Eucl}}  \\
		&=0,	
	\end{align*}
	where in the first inequality we used $\Vol(U \setminus \mathcal{E}_i) \to 0$, and in the last step, we used the assumption that $\partial _j \psi _i \to \partial _j \psi $ in $L^2$ and the fact that uniformly Lipschitz sequence $f _i$ has a locally $W^{1,2}$-weak convergent subsequence.
\end{proof}

\begin{prop}
	$f _\infty$ satisfies the following equation weakly on $\mathbb{R}^3$ :
	\begin{align}\label{f2-eq}
\Delta _{\mathrm{Eucl} } f_\infty ^2 = \frac{24}{m_0^2} \cdot (1-f_\infty)^4.
\end{align}
\end{prop}
\begin{proof}
It's enough to show that for any smooth function $\psi $ on $\mathbb{R}^3$ with compact support in $U \subset B(x_o, D) $, 
\begin{align*}
	-\int_{U} \left<\nabla f^2 _\infty, \nabla \psi\right>_{\mathrm{Eucl} } \mathrm{dvol}_{\mathrm{Eucl}} = \frac{24}{m_0^2} \int_{U} (1-f_\infty)^4 \psi \mathrm{dvol}_{\mathrm{Eucl}} .
\end{align*}
Define $\psi _i := \psi \circ \mathbf{u}_i$ on $\mathcal{E}_i$. Then $\psi _i$ are uniformly Lipschitz and $\psi _i \to \psi $ uniformly.

Notice that
\begin{align*}
	\Delta _{h_i}f_i^2 = 6|\nabla f_i|_{h_i}^2.
\end{align*}
Using integration by parts, for $U_i := \mathbf{u}_i^{-1}U \subset \mathcal{E}_i$, we have
\begin{align*}
	-\int_{U_i} \left< \nabla f _i^2, \nabla \psi _i \right>_{h_i} \mathrm{dvol}_{h_i} &= \int_{U_i} 6 \psi _i |\nabla f_i|^2_{h_i} \mathrm{dvol}_{h_i}  - \int_{\partial U_i} \psi _i \left< \nabla f _i, \vec{n} \right> \mathrm{dA}_{h_i}.
\end{align*}
Since $|\psi _i |, |\nabla f _i|_{h_i} \leq C$ on $U_i$,
$$
\left|\int_{\partial U_i} \psi_i \left<\nabla f _i, \vec{n} \right> \mathrm{dA}_{h_i}\right|\leq C\cdot  \Area_{h_i}(\partial \mathcal{E}_i) \to 0.
$$ 
Using (\ref{grad-f}), we have 
\begin{align*}
	\lim_{i\to \infty} \int_{U_i} 6 \psi _i |\nabla f_i|_{h_i}^2 \mathrm{dvol}_{h_i} &= \lim_{i \to \infty} \int_{U_i} \frac{24}{\mathcal{C}(\Sigma _i, g_i)^2} (1-f_i)^{4} \psi _i \\										  &= \frac{24}{m_0^2} \int_{U}(1-f_\infty)^4 \psi \mathrm{dvol}_{\mathrm{Eucl}} .
\end{align*}
From Lemma \ref{weak-W12}, we know 
$$\lim_{i\to \infty} \int_{U_i} \left< \nabla f^2 _i, \nabla \psi _i\right>_{h_i} \mathrm{dvol}_{h_i} = \int_{U} \left< \nabla f^2 _\infty, \nabla \psi\right>_{\mathrm{Eucl}} \mathrm{dvol}_{\mathrm{Eucl}} ,$$
which concludes the proof. 
\end{proof}

From standard elliptic theory, we know $f_\infty^2 \in C^{1,\alpha }_{loc}(\mathbb{R}^3)$. Also from (\ref{grad-f}), we have
\begin{align*}
	|\nabla f_\infty^2|_{\mathrm{Eucl}} \leq \frac{4 f_\infty(1-f_\infty)^2}{m_0},
\end{align*}
that is, $\left|\nabla (1- f_\infty^2)^{-1}\right|_{\mathrm{Eucl}} \leq C(m_0)$ at any point $x \in \mathbb{R}^3$ such that $f_\infty(x) \neq 1$.
Since $f_\infty(x_o) = \frac{1}{2}$, for any $D>0$ and any $x \in B_{\mathrm{Eucl}}(x_o, D)$, we have
\begin{align*}
	(1- f_\infty^2)^{-1}(x) \leq \frac{4}{3} + C(m_0) D.
\end{align*}
Together with the uniform convergence of $f_i$, the following holds 
\begin{align}\label{f-ball}
	\hat{B}_{h_i, \mathcal{E}_i}(p_i, D) \subset \left\{x \in \tilde{M}_i^* : 1- f_i(x) \geq  \frac{1}{C(m_0)(1+D)} \right\}.
\end{align}

Now we consider another function $\xi _i(x):=\frac{f_i^2}{(1-f_i)^2}$ defined on $(\tilde{M}^*_i, h_i)$, which satisfies $\xi _i(p_i) =1$ and
\begin{align}\label{xi-eq}
	\Delta _{h_i} \xi _i = 6 (1-f_i)^{-4}|\nabla f_i|_{h_i}^2.
\end{align}

For any fixed $D>0$ and any $x \in \hat{B}_{h_i, \mathcal{E}_i}(p_i, D)$, by (\ref{grad-f}) and (\ref{f-ball}), we have 
\begin{align}\label{nbl-xi}
|\nabla \xi _i|_{h_i}(x) = \frac{2 f_i |\nabla f_i|_{h_i}(x)}{(1-f_i)^{3}}\leq C(m_0,D).
\end{align}

Similarly, by the Arzel\`{a}-Ascoli theorem, up to a subsequence, $\xi _i \to \xi _\infty$ locally uniformly and weakly $W^{1,2}$ for some Lipschitz function $\xi _\infty$ on $\mathbb{R}^3$, and $\xi _\infty = \frac{f_\infty^2}{(1-f_\infty)^2}$. We also have 
\begin{align}\label{xi-eq}
		\Delta _{\mathrm{Eucl} } \xi_\infty = \frac{24}{m_0^2},
	\end{align}
	so $\xi _\infty$ is a smooth function. 
Using those properties, we prove the following rigidity:
\begin{lem}\label{lem-rigidity}
$$
\xi _\infty(x) = \frac{4}{m_0^2}|x|^2,
$$ 
and
\begin{align}\label{f-infinity}
	f_\infty(x) = \left( 1 + \frac{m_0}{2|x|} \right) ^{-1}.
\end{align}
\end{lem}
\begin{proof}
	Using the relation $\xi _\infty = \frac{f_\infty^2}{(1-f_\infty)^2}$, (\ref{f2-eq}) and (\ref{xi-eq}), at any point $x$ such that $f_\infty(x) \neq 0$, we have
	\begin{align*}
		\Delta _{\mathrm{Eucl} } \xi _\infty &= (1-f_\infty)^{-3}\Delta _{\mathrm{Eucl}} f_\infty^2 + 6f_\infty(1-f_\infty)^{-4} |\nabla f_\infty|^2\\
		&= \frac{24}{m_0^2} \cdot (1-f_\infty)+ 6f_\infty(1-f_\infty)^{-4} |\nabla f_\infty|^2\\
		&= \frac{24}{m_0^2},
	\end{align*}
	which implies that $|\nabla f_\infty|^2= \frac{4}{m_0^2}\cdot (1-f_\infty)^4$. So 
	$$|\nabla \xi _\infty|^2 = 4(1-f_\infty)^{-6}f_\infty^2|\nabla f_\infty|^2= \frac{16}{m_0^2} f_\infty^2(1-f_\infty)^{-2} = \frac{16}{m_0^2} \xi _\infty,$$
	which also holds at any $ x \in \mathbb{R}^3$ since $\xi _\infty$ is smooth. 

	From the Bochner formula, we have
	\begin{align*}
		|\nabla ^2 \xi _\infty - \frac{8}{m_0^2} \delta |^2 =0.
	\end{align*}
Consider the smooth function $\eta (x):= \xi _\infty(x) - \frac{4}{m_0^2}|x|^2$. Then $\nabla ^2 \eta =0$, i.e. $\eta $ is a linear function, and $\eta (x_o) =0$. Assuming that $\eta (x) = (x-x_o)\cdot b$ for some $b \in \mathbb{R}^3$. Using $|\nabla \xi _\infty |^2 = \frac{16}{m_0^2} \xi _\infty$, we have $|b|^2 + \frac{16}{m_0^2} x_o \cdot b =0$. So 
\begin{align*}
	\xi _\infty(x) = \frac{4}{m_0^2} \left| x + \frac{m_0^2}{8}b\right|^2.
\end{align*}
Up to a translation of the coordinates, we can take $x_o = (\frac{m_0}{2}, 0, 0)$ and $\xi (x) = \frac{4}{m_0^2}|x|^2$.
\begin{comment}
By taking limit of (\ref{xi-quad}), $\eta $ has quadratic growth, so $\eta $ is a polynomial of degree at most $2$. Since $\xi _\infty \geq 0$ and $\xi _\infty(x_o)= \xi _\infty( (\frac{m_0}{2}, 0, 0) )=1$, we can assume $\xi _\infty$ has the following form:
$$
\xi _\infty(x) = c_{11} \left(x_1- \frac{m_0}{2}\right) ^2 + c_{22} x_2^2 + c_{33} x_3^2 + \frac{4}{m_0^2}|x|^2,
$$ 
where $c_{11}, c_{22}, c_{33}$ are constants to be determined such that $c_{11} + \frac{4}{m_0^2} \geq 0, c_{22} + \frac{4}{m_0^2} \geq 0, c_{33}+ \frac{4}{m_0^2} \geq 0$ and $c_{11}+ c_{22}+ c_{33}=0$.
Taking limit of (\ref{nbl-xi}), $|\nabla \xi _\infty|_{\delta }^2 \leq \frac{16}{m_0^2} \xi _\infty $, which implies that for any $x_1 \in \mathbb{R}$,
\begin{align*}
	\left( c_{11}(x_1 - \frac{m_0}{2}) + \frac{4}{m_0^2}x_1 \right)^2 &\leq \frac{4}{m_0^2}\left( c_{11}(x_1 - \frac{m_0}{2})^2 + \frac{4}{m_0^2} x_1^2 \right)  
\end{align*}
from which we can get that $c_{11}=0$. So we also have that for any $x_2, x_3 \in \mathbb{R}$,
\begin{align*}
	\left( c_{22} + \frac{4}{m_0^2} \right) ^2 x_2^2 &\leq \frac{4}{m_0^2} (c_{22}+ \frac{4}{m_0^2}) x_2^2 \\
\left( c_{33} + \frac{4}{m_0^2} \right) ^2 x_3^2 &\leq \frac{4}{m_0^2} (c_{33}+ \frac{4}{m_0^2}) x_3^2 ,
\end{align*}
from which we can get that $c_{22},c_{33}\leq 0$, and thus $c_{22}=c_{33}=0$ by $c_{22}+ c_{33}=0$. 
\end{comment}
\end{proof}

\section{Proof of main theorems}
We first prove the stability of the mass-capacity inequality.
\begin{proof}[Proof of Theorem \ref{mass-cap-stability}]
	Assume that $(M_i^3, g_i)$ is a sequence of complete one-ended asymptotically flat $3$-manifolds with nonnegative scalar curvature and compact connected outermost horizon boundaries $\Sigma_i$. Suppose that for some constant $m_0>0$, both $m(g_i) \to m_0$ and $  \mathcal{C}(\Sigma_i, g_i) \to m_0$. 

By Proposition \ref{doubling-approx}, without loss of generality, we can assume that there exists a doubling $(\tilde{M}_i, g_i)$ of each $(M_i, g_i)$ such that $(\tilde{M}_i, g_i)$ is symmetric about minimal surface $\Sigma_i$, has nonnegative scalar curvature and two harmonically flat ends. Let $f_i$ be harmonic functions defined by (\ref{Green-2end}), $h_i= f_i^{4} g_i$ and $\tilde{M}^*_i$ the compactification. Denote by $\varepsilon _i = m(h_i) \to 0$. Using the same notations in last section, we have proved that there exist smooth regions $\mathcal{E}_i \subset \tilde{M}^*_i$ such that $\Area_{h_i}(\partial \mathcal{E}_i) \to 0$, and for base point $p_i \in \mathcal{E}_i \cap \Sigma_i$,
$$
(\mathcal{E}_i, \hat{d}_{h_i, \mathcal{E}_i}, p_i) \xrightarrow{\mathrm{pm-GH}} (\mathbb{R}^3, d_{\mathrm{Eucl}}, x_o),
$$ 
where $x_o=(\frac{m_0}{2}, 0,0)$
and 
$$f_i \to f_\infty(x) = \left( 1+ \frac{m_0}{2|x|} \right) ^{-1}$$
locally uniformly.

	In the following, we will only consider the manifold $M_i$, which is equivalent to $\{ \frac{1}{2} \leq f_i < 1\} $. Since $h_i = f_i^{4}g_i$, we know $g_i$ and $h_i$ are two uniformly equivalent metrics on $M_i$. When discussing uniform upper or lower bound on volume, area, or distance, etc., there is no difference between using either metric, so we will omit the subscript for simplicity. Additionally, we always identify $\mathcal{E}_i$ as a subset of $\mathbb{R}^3$  using the diffeomorphism $\mathbf{u}_i$. 

Notice that on $\{f_\infty \geq \frac{1}{2}\} $, the standard Schwarzschild metric is given by
\begin{align*}
	g_{\mathrm{Sch}} = f_\infty^{-4} g_{\mathrm{Eucl}}.
\end{align*}
Since $g_i = f_i^{-4} h_i$, for any $D>0$ and any $x \in \hat{B}_{h_i, \mathcal{E}_i}(p_i, D) \cap \{f_i \geq \frac{1}{2}\} $, we have
\begin{align*}
	\| g_i - g_{\mathrm{Sch}} \|(x) \leq \Psi (\varepsilon _i| D).
\end{align*}

Fix $0< \epsilon \ll 1$. By co-area formula,
\begin{align*}
	\int_{\frac{1}{2}+ 2\epsilon }^{\frac{1}{2}+ 4\epsilon } \mathrm{Length}(\partial \mathcal{E}_i \cap \{f_\infty = t\}) dt \leq \int_{\partial \mathcal{E}_i \cap \{\frac{1}{2} \leq f_\infty \leq \frac{3}{4} \} } |\nabla f_\infty| \leq C \cdot \Area(\partial \mathcal{E}_i).
\end{align*}
So there exists a regular value $t_0 \in (\frac{1}{2}+ 2\epsilon , \frac{1}{2}+ 4\epsilon )$ such that $\partial \mathcal{E}_i \cap \{f_\infty = t_0\} $ consists of smooth curves, and the total length satisfies 
\begin{align*}
	\mathrm{Length}(\partial \mathcal{E}_i \cap \{f_\infty = t_0\} ) \leq \frac{C \cdot \Area(\partial \mathcal{E}_i)}{\epsilon } = \Psi (\varepsilon _i| \epsilon ).
\end{align*}
We can assume $t_0 = \frac{1}{2}+ 3\epsilon $ for simplicity. 
By the uniform convergence $f_i \to  f_\infty$, for all large enough $i$, we have
$$
\mathcal{E}_i \cap \{\frac{1}{2}+ 3\epsilon \leq f_\infty \leq 1 - \frac{1}{\epsilon }\} \subset M_i.
$$
Define $\mathcal{E}_i(\epsilon )$ as the noncompact component of  $$ \mathcal{E}_i \cap \{f_\infty \geq \frac{1}{2}+ 3\epsilon \} \cap M_i.$$

We make some modifications on $ \mathcal{E}_i(\epsilon )$ for later usage as in the proof of \cite[Lemma 4.3]{DS'23}. Let $\{D_k\} _{k=0}^{N} \subset \mathcal{E}_i$ be the components of $\{f_\infty= \frac{1}{2}+ 3\epsilon \} \cap \mathcal{E}_i $, and assume $D_0$ has the largest area. Then $ \sum_{k \geq 1}^{} \mathrm{diam}D_k \leq \Psi (\varepsilon _i| \epsilon )$. Choose $x_k \in D_k$ for each $k\geq 1$. Notice that the Euclidean Hausdorff distance between $D_0$ and $D_k$ is bounded by $ \Psi (\varepsilon _i|\epsilon )$. Then there exists $y_k \in D_0$ such that $\hat{d}_{h_i, \mathcal{E}_i}(x_k, y_k) \leq \Psi (\varepsilon _i| \epsilon )$. For each $k \geq 1$, let $\gamma _k$ be a geodesic between $x_k, y_k$ for the metric $\hat{d}_{h_i, \mathcal{E}_i}$. By thickening each $\gamma _k$, we can get thin solid tubes $T_k$ inside $\Psi (\varepsilon _i|\epsilon )$-neighborhood of $\gamma _k$ with arbitrarily small boundary area. Let $\mathcal{E}_i(\epsilon )' := \mathcal{E}_i(\epsilon ) \cup (\cup _k T_k).$ We then get a smooth connected subset by smoothing corners of $\mathcal{E}_i(\epsilon )'$.

For simplicity, we still denote the modified $\mathcal{E}_i(\epsilon )'$ by $\mathcal{E}_i(\epsilon )$. For all large enough $i$, we have $\mathcal{E}_i(\epsilon ) \subset M_i \cap \mathcal{E}_i$. Denote by 
$$\Sigma _i(\epsilon ) = \overline{\partial \mathcal{E}_i(\epsilon ) \setminus \partial \mathcal{E}_i}.$$
Then $\Sigma _i(\epsilon )$ is a connected closed surface with boundary, whose total length is bounded by $\Psi (\varepsilon _i|\epsilon )$, and $\partial \mathcal{E}_i(\epsilon ) = \Sigma _i(\epsilon ) \cup \left(  \partial \mathcal{E}_i \cap \{f_\infty \geq \frac{1}{2} + 3 \epsilon \} \right) $. Notice that $\Sigma _i(\epsilon )$ lies inside $\Psi (\varepsilon _i| \epsilon )$-neighborhood of $\{f_\infty = \frac{1}{2}+ 3 \epsilon \} \cap \mathcal{E}_i$, which is contained in $(\Psi (\varepsilon _i|\epsilon ) + C\cdot \epsilon )$-neighborhood of $\{f_i = \frac{1}{2}\} \cap \mathcal{E}_i$ with respect to both distances $\hat{d}_{\mathrm{Eucl}}$ and $\hat{d}_{g_i}$. See also the Figure \ref{fig2}.

\begin{figure}[htpb]
	\centering
	\includegraphics[width=0.8\textwidth]{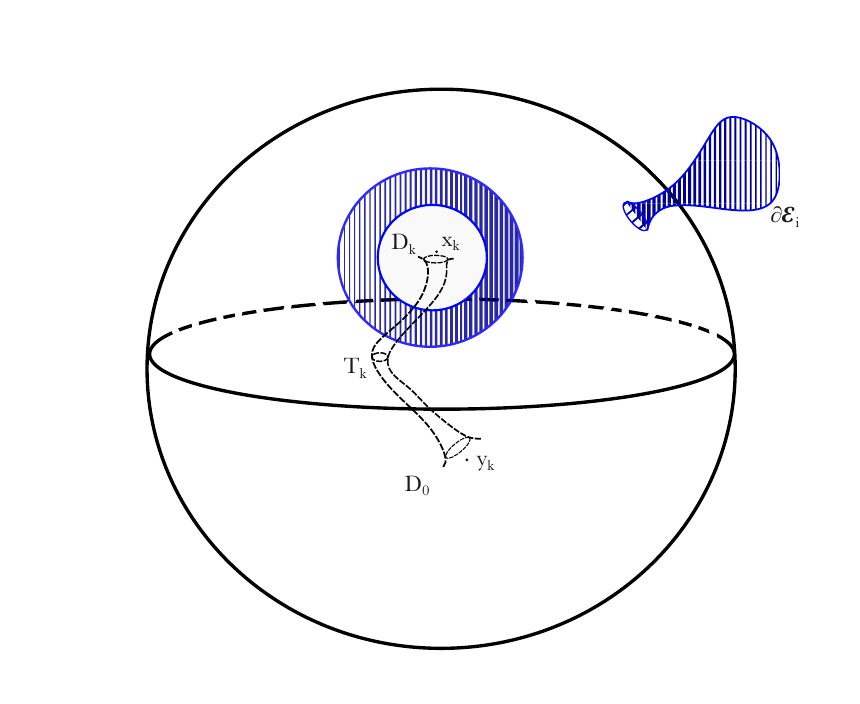}
	\caption{This picture shows the constructed region $\mathcal{E}_i(\epsilon )$ near $\Sigma _i(\epsilon )$. The sphere is $\{f_\infty = \frac{1}{2} + 3 \epsilon \} $, and the blue shaded region is included in the complement of $\mathcal{E}_i$. $\partial \mathcal{E}_i(\epsilon )$ consists of two parts: one is $\partial \mathcal{E}_i \cap \{f_\infty \geq \frac{1}{2}+ 3 \epsilon \} $, the boundary of the blue region; the other is $\Sigma _i(\epsilon )$, which is connected and consists of the connected sum of $D_0$ and $D_k$ using the boundary of the tubes $T_k$. }
	\label{fig2}
\end{figure}

Denote by $M_{\mathrm{Sch}}(\epsilon ) := \{f_\infty \geq \frac{1}{2}+ 3\epsilon \} $, and $x_o(\epsilon ) \in \{ f_\infty = \frac{1}{2}+ 3\epsilon \} $ a base point closest to $x_o$. 
\begin{prop}\label{eps-GH}

	\begin{align}\label{total}
		(\mathcal{E}_{i}(\epsilon ), \hat{d}_{g_i, \mathcal{E}_i(\epsilon )}, q_i) \to (M_{\mathrm{Sch}}(\epsilon ), d_{g_{\mathrm{Sch}}}, x_o(\epsilon ) )
	\end{align}
	in the pointed measured Gromov-Hausdorff topology for some base points $q_i \in \Sigma _i(\epsilon )$, and
	\begin{align}\label{boundary}
		(\Sigma _i(\epsilon ), \hat{d}_{g_i, \Sigma _i(\epsilon )}) \to (\partial M_{\mathrm{Sch}}(\epsilon ), \hat{d}_{g_{\mathrm{Sch}}, \partial M_{\mathrm{Sch}}(\epsilon )} )
	\end{align}
	in the measured Gromov-Hausdorff topology for the induced length metrics.
\end{prop}
\begin{proof}
	We firstly show (\ref{boundary}). By the construction of $\mathcal{E}_i(\epsilon )$,
	we know $\Sigma _i(\epsilon )$ is inside $\Psi (\varepsilon _i| \epsilon )$-neighborhood of $\partial \mathcal{E}_i(\epsilon ) \cap \partial M_{\mathrm{Sch}}(\epsilon )$, so
		$$
		(\Sigma _i(\epsilon ), \hat{d}_{g_{\mathrm{Sch}}, \Sigma _i(\epsilon )} ) \to (\partial M_{\mathrm{Sch}}(\epsilon ), \hat{d}_{g_{\mathrm{Sch}}, \partial M_{\mathrm{Sch}}(\epsilon )})
		$$
		in the Gromov-Hausdorff topology. It remains to show that for any $x, y \in \Sigma _i(\epsilon )$,
	$$ \lim_{i\to \infty}\hat{d}_{g_i, \Sigma _i(\epsilon )}(x,y) = \hat{d}_{g_{\mathrm{Sch}}, \partial M_{\mathrm{Sch}}(\epsilon )}(x,y). $$
	Let $\gamma \subset \partial M_{\mathrm{Sch}}(\epsilon ) $ be a geodesic between $x$ and $y$ for distance $\hat{d}_{g_{\mathrm{Sch}}, \partial M_{\mathrm{Sch}}(\epsilon ) }$. By our construction of $\mathcal{E}_i(\epsilon )$, we can always perturb $\gamma $ to get $\tilde{\gamma } \subset \partial \mathcal{E}_i(\epsilon ) \setminus \partial \mathcal{E}_i$ such that $|\mathrm{Length}_{\mathrm{Sch}}(\gamma ) - \mathrm{Length}_{\mathrm{Sch}}(\tilde{\gamma })| \to 0$. Then 
	\begin{align*}
		\hat{d}_{g_i, \Sigma _i(\epsilon )}(x,y) &\leq \int_{0}^{1}|\tilde{\gamma }'|_{g_i} \\
	&\leq (1+ \Psi (\varepsilon _i) ) \int_{0}^{1}|\tilde{\gamma }'|_{g_{\mathrm{Sch}}} .
	\end{align*}
	Taking $i\to \infty$, we have
	\begin{align*}
		\lim_{i\to \infty}\hat{d}_{g_i, \Sigma _i(\epsilon )}(x,y) \leq \hat{d}_{g_{\mathrm{Sch}}, \partial M_{\mathrm{Sch}}(\epsilon )}(x,y).
	\end{align*}
	Similarly, it's easy to check 
	\begin{align*}
		\hat{d}_{g_{\mathrm{Sch}}, \partial M_{\mathrm{Sch}}(\epsilon )}(x,y) \leq \lim_{i\to \infty} \hat{d}_{g_i, \Sigma _i(\epsilon )}(x,y).
	\end{align*}
	This completes the Gromov-Hausdorff convergence in (\ref{boundary}).

	Notice that for any $x, y \in \partial M_{\mathrm{Sch}}(\epsilon )$, there is a uniform constant $C>0$ such that 
	\begin{align*}
		d_{\mathrm{Eucl}}(x,y) \leq \hat{d}_{g_{\mathrm{Sch}}, \partial M_{\mathrm{Sch}}(\epsilon )}(x, y) \leq C d_{\mathrm{Eucl}}(x,y),
	\end{align*}
	which implies that, for any $x, y \in \Sigma _i(\epsilon )$, for all large enough $i$,
	\begin{align}\label{bd-vs-total}
		\frac{1}{2}\hat{d}_{h_i, \mathcal{E}_i}(x,y) \leq \hat{d}_{g_i, \Sigma _i(\epsilon )}(x,y) \leq C\hat{d}_{h_i, \mathcal{E}_i}(x,y)+ \Psi (\varepsilon _i|\epsilon ).
	\end{align}

	Now we prove (\ref{total}). It is enough to show that for any fixed $D>0$ and any $x, y \in \mathcal{E}_i(\epsilon ) \cap B(q_i, D)$, 
	$$|\hat{d}_{g_i, \mathcal{E}_i(\epsilon )}(x,y) - d_{g_{\mathrm{Sch}}}(x,y) | \to 0 \text{ as } i\to \infty. $$
	It's easy to check that $d_{g_{\mathrm{Sch}}}(x,y) \leq \lim_{i\to \infty} \hat{d}_{g_i, \mathcal{E}_i(\epsilon )}(x,y)$. In the following, we will show the other inequality.
	
	For any fixed $\delta _0>0$. We firstly assume that $d_{g_{\mathrm{Sch}}}(x,y) \geq \delta _0$, $d_{g_{\mathrm{Sch}}} (x, \partial M_{\mathrm{Sch}}(\epsilon ) ) \geq \delta _0$ and $d_{g_{\mathrm{Sch}}}(y, \partial M_{\mathrm{Sch}}(\epsilon ) ) \geq \delta _0$. If $\gamma $ is a $g_{\mathrm{Sch}}$-geodesic between $x, y$, then by the geometry of Schwarzschild metric, $d_{g_{\mathrm{Sch}}}(\gamma , \partial M_{\mathrm{Sch}}(\epsilon ) ) \geq \delta _0$. For any $0< \delta \ll\delta _0$, we can use piecewise line segments $\{\gamma _j\} _{j=1}^{N(\delta ,D)} $ to approximate $\gamma $ such that $ \Sigma_{j=1}^{N} \mathrm{Length}_{g_{\mathrm{Sch}}}(\gamma _j) \leq  \mathrm{Length}_{g_{\mathrm{Sch}}}(\gamma )  + \Psi (\delta) ,$ and $\mathrm{Length}_{\mathrm{Eucl}}(\gamma _j) \geq \delta $. For each $\gamma _j$ with $\gamma _j(0)= x_j,  \gamma _j(1) = y_j$, from the proof of \cite[Lemma 4.5]{DS'23}, we can find peturbed points $x_j', y_j'$ such that the straight line segment $\tilde{\gamma} _j$ between $x_j', y_j'$ lies in $\mathcal{E}_i(\epsilon )$ and $d_{\mathrm{Eucl}}(x_j, x_j') + d_{\mathrm{Eucl}}(y_j, y_j') \leq \Psi (\varepsilon _i)$. Then
\begin{align*}
	\hat{d}_{g_i, \mathcal{E}_i(\epsilon )}(x,y) &\leq \sum_{j} \hat{d}_{g_i, \mathcal{E}_i(\epsilon )}(x_j', y_j') + C \sum_{j} ( d_{\mathrm{Eucl}}(x_j, x_j') + d_{\mathrm{Eucl}}(y_j, y_j') )\\
						     &\leq \sum_{j} \mathrm{Length}_{g_i} (\tilde{\gamma} _j) +  \Psi (\varepsilon _i| \delta , D )\\
						     & \leq (1+ \Psi (\varepsilon _i|D) ) \sum_{j}\mathrm{Length}_{g_{\mathrm{Sch}}}(\gamma_j) + \Psi (\varepsilon _i| \delta ,D).
\end{align*}
Taking $i\to \infty$, we have 
\begin{align*}
	\lim_{i\to \infty}\hat{d}_{g_i, \mathcal{E}_i(\epsilon )}(x,y) \leq d_{g_{\mathrm{Sch}}}(x,y) + \Psi (\delta ).
\end{align*}
Taking $\delta \to 0$ gives 
\begin{align}\label{dlt-0-dis}
\lim_{i\to \infty}\hat{d}_{g_i, \mathcal{E}_i(\epsilon )}(x,y) \leq d_{g_{\mathrm{Sch}}}(x,y).
\end{align}

If $d_{g_{\mathrm{Sch}}}(x,y) \leq \delta _0$ and $d_{g_{\mathrm{Sch}}}(x, \partial M_{\mathrm{Sch}}(\epsilon ) ) \leq \delta _0$, then $d_{\mathrm{Eucl}}(x, \partial M_{\mathrm{Sch}}(\epsilon ) ) \leq \delta _0$ and $d_{\mathrm{Eucl}}(x,y) \leq \delta _0$ . We can take an almost $\hat{d}_{h_i, \mathcal{E}_i}$-geodesic $\gamma \subset \mathcal{E}_i $ between $x,y$, so $\mathrm{Length}_{h_i}(\gamma ) \leq 2 \delta _0$. If $\gamma \subset \mathcal{E}_i(\epsilon )$, we have $\hat{d}_{g_i, \mathcal{E}_i(\epsilon )}(x,y) \leq C \delta _0$; otherwise, let $x'$ be the first intersection point of $\gamma $ and $\Sigma _i(\epsilon )$ and $y'$ be the last intersection point. By (\ref{bd-vs-total}), we have
$$\hat{d}_{g_i, \mathcal{E}_i(\epsilon )}(x', y') \leq \hat{d}_{g_i, \Sigma _i(\epsilon )}(x',y') \leq C \hat{d}_{h_i, \mathcal{E}_i}(x',y') + \Psi (\varepsilon _i|\epsilon ) \leq C \delta _0 + \Psi (\varepsilon _i|\epsilon ).$$ 
So 
$$\hat{d}_{g_i, \mathcal{E}_i(\epsilon )}(x,y) \leq \hat{d}_{g_i, \mathcal{E}_i(\epsilon )}(x,x') + \hat{d}_{g_i, \mathcal{E}_i(\epsilon )}(x',y') + \hat{d}_{g_i, \mathcal{E}_i(\epsilon )}(y',y) \leq C \delta _0 + \Psi (\varepsilon _i|\epsilon ).$$
This shows that the pointed Gromov-Hausdorff distance
\begin{align*}
	d_{pGH}((\mathcal{E}_i(\epsilon ), \hat{d}_{g_i, \mathcal{E}_i(\epsilon )}, q_i), (\mathcal{E}_i(\epsilon ) \cap \{ x: d_{g_{\mathrm{Sch}}}(x, \partial M_{\mathrm{Sch}}(\epsilon ) )\geq \delta _0\} , &\hat{d}_{g_i, \mathcal{E}_i(\epsilon )}, q_i) )\\
	&\leq C \delta _0 + \Psi (\varepsilon _i|\epsilon ).
\end{align*}
Together with (\ref{dlt-0-dis}), we have 
\begin{align*}
	d_{pGH}( (\mathcal{E}_i(\epsilon ) , \hat{d}_{g_i, \mathcal{E}_i(\epsilon )}, q_i), ( M_{\mathrm{Sch}}(\epsilon ), d_{g_{\mathrm{Sch}}}, x_o(\epsilon ) ) ) \leq C \delta _0 + \Psi (\varepsilon _i| \epsilon ).
\end{align*}
Choosing a sequence $\delta _0 \to 0$ and by a diagonal argument, for a subsequence $i\to \infty$, we get the conclusion on the pointed Gromov-Hausdorff convergence in (\ref{total}). 

Since the Hausdorff measure induced by $\hat{d}_{g_i, \mathcal{E}_i(\epsilon )}$ and $\hat{d}_{g_i, \Sigma_i(\epsilon )}$ are the same as the volume element $\mathrm{dvol}_{g_i}$ and the area element $\mathrm{dA}_{g_i}$ respectively, together with the isoperimetric inequality, it's standard to check that these measures also converge weakly (cf. \cite[Page 22]{DS'23}). In particular, $\Area_{g_i}(\Sigma _i(\epsilon ) ) \to \Area_{g_{\mathrm{Sch}}}(\partial M_{\mathrm{Sch}}(\epsilon ) )$.

\end{proof}
\begin{comment}
By our construction, for all large enough $i$, using the uniform convergence of $f_i \to f_\infty$, the Euclidean Hausdorff distance between $\mathcal{E}_i(\epsilon )$ and $M_i \cap \mathcal{E}_i$ is bounded by $C\cdot \epsilon $, so
\begin{align*}
	d_{pGH}\left( (\mathcal{E}_i(\epsilon _i), \hat{d}_{g_i, \mathcal{E}_i(\epsilon _i)}, q_i) , (M_i \cap \mathcal{E}_i, \hat{d}_{g_i}) \right) \leq C\cdot \epsilon ,
\end{align*}
\end{comment}

Choosing a sequence $\epsilon _i \to 0$, and using a diagonal argument, we can take a subsequence such that 
\begin{align*}
	(\mathcal{E}_i(\epsilon _i), \hat{d}_{g_i, \mathcal{E}_i(\epsilon _i)}, q_i) \to (M_{\mathrm{Sch}}, d_{g_{\mathrm{Sch}}}, x_o)
\end{align*}
in the pointed measured Gromov-Hausdorff topology, and 
\begin{align*}
	(\Sigma _i(\epsilon _i), \hat{d}_{g_i, \Sigma _i(\epsilon _i)}) \to (\partial M_{\mathrm{Sch}}, \hat{d}_{g_{\mathrm{Sch}}, \partial M_{\mathrm{Sch}}})	
\end{align*}
in the measured Gromov-Hausdorff topology.

We take $E_i = \mathcal{E}_i(\epsilon _i)$, $\Sigma _i^1 = \Sigma _i(\epsilon _i)$ and $\Sigma _i^2 = \partial \mathcal{E}_i \cap \{f_\infty \geq \frac{1}{2} + 3 \epsilon _i\} $. This completes the proof.

\end{proof}

Finally, we prove the stability of the Penrose inequality as a corollary of the stability of the mass-capacity inequality.

\begin{proof}[Proof of Theorem \ref{Penrose-stability}]
	The case when $A_0=0$ was proved in \cite{DS'23}. Let $A_0 > 0$ be a fixed constant and $(M_i^{3}, g_i)$ be a sequence of one-ended asymptotically flat $3$-manifolds with nonnegative scalar curvature, whose boundaries are compact outermost minimal surfaces with area $A_i \to A_0$. Suppose that the ADM mass $m(g_i)$ converges to $ \sqrt{\frac{A_0}{16 \pi }} $.

	As proved in \cite{Bray'01} (see also the remark at the end of Section \ref{capacity}), for each $i$, we can find a sequence of conformal metrics $g_i(t)$  and smooth subset $M_i(t) \subset M_i$ such that $(M_i(t), g_i(t))$ is an one-ended asymptotically flat $3$-manifold with nonnegative scalar curvature and an outermost minimal boundary $\Sigma _i(t)$. From \cite[Theorem 4]{Bray'01}, for some large enough $T_i>0$, for all $t>T_i$, $M_i(t)$ is diffeomorphic to $\mathbb{R}^3 \setminus \{0\} $. In particular, $\Sigma _i(t)$ is connected. Moreover, 
	$$\mathrm{Area}_{g_i(t)}(\Sigma _i(t) ) = A_i,\ \  m(g_i ) \geq m(g_i(t)) \to  \sqrt{\frac{A_i}{16 \pi }} \text{ as } t\to \infty .$$

From \cite[Section 7]{Bray'01}, we know that for a.e. $t$, 
$$m'(g_i(t)) = 2\left(\mathcal{C}(\Sigma_i(t), g_{i}(t ) ) - m(g_i(t) ) \right) \leq 0.$$
So we can find $t_i \in (T_i, T_i+1)$ such that 
\begin{align*}
	2\left( m(g_i(t_i) ) - \mathcal{C}(\Sigma_i (t_i), g_i(t_i) ) \right) \leq m(g_i( T_i) ) - m(g_i (T_i+1) ) \leq m(g_i) - \sqrt{\frac{A_i}{16 \pi }} .
\end{align*}

Then we can apply Theorem \ref{mass-cap-stability} to $(M_i(t_i), g_i(t_i) )$. The conclusion follows if we take $N_i$ to be the good region $E_i$ associated with $(M_i(t_i), g_i(t_i) )$.

\end{proof}

\appendix

\section{GH convergence modulo negligible domains in general dimensions}\label{appendix}
In this appendix, we provide a detailed proof of Theorem \ref{GH-spikes} for reader's convenience. In fact, we will prove a stronger quantitative version of the theorem. The proof involves a slight modification of the arguments presented in \cite{DS'23}, supplemented by an induction argument.

To make the proof clearer, we will prove the theorem for subsets of $\mathbb{R}^n$ in the following. Note that the same proof also applies to the general case (see also the remark at the end of this section). 

\begin{thm}
	Assume that $n \geq 2$, $Y_i \subset \mathbb{R}^n$ is a sequence of domains with smooth compact boundaries $\Sigma _i$, $Y_i$ contains the end of $\mathbb{R}^n$, and  $\mathcal{H}^{n-1}(\Sigma _i) \to 0$, then there exists a sequence of smooth closed subsets $Y_i'' \subset Y_i$ such that $\mathcal{H}^{n-1}(\partial Y_i'') \to 0$ and for any base point $q_i \in Y_i''$,
	\begin{align*}
		(Y_i'', q_i, \hat{d}_{\mathrm{Eucl}, Y_i''}) \xrightarrow{pmGH} (\mathbb{R}^n, 0, \hat{d}_{\mathrm{Eucl}}). 
	\end{align*}
Moreover, we have a quantitative version: there exists $ 0< \varepsilon (n) \ll 1$ so that if $Y \subset \mathbb{R}^n$ is a domain with smooth compact boundary $\Sigma $, $Y$ contains the end of $\mathbb{R}^n$, and $\mathcal{H}^{n-1}(\Sigma ) \leq \varepsilon(n) $, then we can find a perturbation $Y'' \subset Y$, such that 
\begin{align}
	\mathcal{H}^{n-1}(\partial Y'') \leq (\mathcal{H}^{n-1}(\Sigma ) ) ^{ 1- 10^{-4}n^{-1}},
\end{align}
and for any base point $q \in Y''$,
\begin{align}\label{explicit error}
	d_{pGH}( (Y'', \hat{d}_{Y''}, q), (\mathbb{R}^n, \hat{d}_{\mathrm{Eucl}},0) ) \leq (\mathcal{H}^{n-1}(\Sigma ) )^{2^{-n}}.
\end{align}
\end{thm}
\begin{proof}
	Notice that the case when $n=2$ is obvious by definition, and the case when $n=3$ has been proved in \cite{DS'23}. So we assume that $n \geq 4$. In the following, we will prove the quantitative version by induction and assume that the conclusion holds for all dimensions less than or equal to $n-1$.

We will find out $\varepsilon (n)$ by induction. 
Set $\eta := \mathcal{H}^{n-1}(\Sigma )\leq  \varepsilon (n) \ll 1$.
Let $\mathcal{W}$ be the domain bounded by $\Sigma $. By the isoperimetric inequality,
\begin{align*}
	\mathcal{H}^n(\mathcal{W}) \leq C(n) \mathcal{H}^{n-1}(\Sigma )^{\frac{n}{n-1}}\leq C(n) \eta^{\frac{n}{n-1}} \leq \eta .
\end{align*}
Take $\delta _0 = \eta^{10^{-2}n^{-1}},\delta_1 = \eta ^{10^{-4} n^{-1}}$. For any $\mathbf{k}= (k_1, k_2, \cdots, k_n) \in \mathbb{Z}^n,$ consider the closed cube $\mathbf{C}_{\mathbf{k}}(\delta_1 )$ defined by
\begin{align*}
	\mathbf{C}_{\mathbf{k}}(\delta_1 ):= [k_1\delta_1 , (k_1+1)\delta_1 ] \times \cdots \times [k_n \delta_1 , (k_n+1)\delta_1 ] \subset \mathbb{R}^n.
\end{align*}

For $t\in\mathbb{R}$, define the plane $$A_{\mathbf{k}, \delta _1}(t):= \{(x_1, \cdots, x_n) \in \mathbb{R}^n: x_n= (k_n +t) \delta _1\}.$$ 
By definition $\mathbf{C}_{\mathbf{k}}(\delta _1) \subset  \bigcup_{t\in[0,1]} A_{\mathbf{k}, \delta _1}(t)$.

By the coarea formula, there exists $t_{\mathbf{k}} \in (\frac{1}{2}, \frac{1}{2}+ \delta _0)$ such that $ \Sigma ^{n-2}(t_{\mathbf{k}}):= A_{\mathbf{k}, \delta _1}(t_{\mathbf{k}}) \cap \Sigma \cap \mathbf{C}_{\mathbf{k}}(\delta _1)$ consists of $(n-2)$-submanifolds and 
\begin{align}\label{small length k-n}
\begin{split}
	\mathcal{H}^{n-2}(\Sigma^{n-2}(t_{\mathbf{k}})) &\leq \frac{\mathcal{H}^{n-1}(\Sigma \cap \mathbf{C}_{\mathbf{k}}(\delta _1) )}{\delta _0 \delta _1}\\
											      &\leq 4 \eta ^{1-  (10n)^{-1}}.
\end{split}
\end{align}
If $A_{\mathbf{k}, \delta _1}(t_{\mathbf{k}}) \cap \Sigma \cap \partial \mathbf{C}_{\mathbf{k}}(\delta _1) \neq \emptyset$, then there exist components $ \mathcal{P}$ of $A_{\mathbf{k}, \delta _1}(t_{\mathbf{k}})\cap \partial \mathbf{C}_{\mathbf{k}}(\delta _1) \setminus \Sigma $ and a bounded open subset $Z(t_{\mathbf{k}}) \subset A_{\mathbf{k}, \delta _1}(t_{\mathbf{k}}) \cap \mathbf{C}_{\mathbf{k}}(\delta _1)$ such that $\mathcal{H}^{n-2}(\mathcal{P}) \leq 2(n-1) \cdot \mathcal{H}^{n-2}(\Sigma ^{n-2}(t_{\mathbf{k}}))$, $\partial Z(t_{\mathbf{k}}) \subset \mathcal{P} \cup \Sigma ^{n-2}(t_{\mathbf{k}})$, and $\Sigma ^{n-2}(t_{\mathbf{k}}) \subset \overline{Z(t_{\mathbf{k}})}$. In particular, 
\begin{align}\label{Zk boundary}
	\mathcal{H}^{n-2}(\partial Z(t_{\mathbf{k}}) ) \leq 2n \cdot \mathcal{H}^{n-2}(\Sigma ^{n-2}(t_{\mathbf{k}}) ) \leq 8n\cdot \eta ^{1- (10n)^{-1}}.
\end{align}

Take $\varepsilon (n)$ small enough such that $8n\cdot \varepsilon (n)^{1- (10n)^{-1}} \leq \varepsilon (n-1)$. 
Applying the induction assumption to $A_{\mathbf{k}, \delta _1}(t_{\mathbf{k}}) \setminus Z(t_{\mathbf{k}})$ and using (\ref{Zk boundary}), 
we can find a perturbation $\tilde{\Sigma }^{n-2}(t_{\mathbf{k}}) \subset A_{\mathbf{k}, \delta _1}(t_{\mathbf{k}})$ of $\partial Z(t_{\mathbf{k}}) $ satisfying the following properties: 
\begin{itemize}
	\item let $\tilde{Z}(t_{\mathbf{k}}) \subset A_{\mathbf{k}, \delta _1}(t_{\mathbf{k}})$ be the domain bounded by $\tilde{\Sigma }^{n-2}(t_{\mathbf{k}})$, then $Z(t_{\mathbf{k}}) \subset \tilde{Z}(t_{\mathbf{k}})$, and $A_{\mathbf{k}, \delta _1}(t_{\mathbf{k}}) \setminus \tilde{Z}(t_{\mathbf{k}})$ contains the end of $\mathbb{R}^{n-1}$; 
	\item the area satisfies  \begin{align}\label{area-ind-n}
			\begin{split}
				\mathcal{H}^{n-2}(\tilde{\Sigma }^{n-2}(t_{\mathbf{k}}) )& \leq \mathcal{H}^{n-2}(\partial Z(t_{\mathbf{k}}) ) ^{1- 10^{-4}(n-1)^{-1}}\\
	&\leq ( 2 n \cdot \mathcal{H}^{n-2}( \Sigma ^{n-2}(t_{\mathbf{k}}) ) )^{1- 10^{-4}(n-1)^{-1}};
\end{split}
\end{align}
\item for any two points $x, y \in \tilde{Y}(t_{\mathbf{k}}):= A_{\mathbf{k}, \delta _1}(t_{\mathbf{k}}) \setminus \tilde{Z}(t_{\mathbf{k}})$, we have
	\begin{align}\label{Y-tld}
	|\hat{d}_{\tilde{Y}(t_{\mathbf{k}})}(x,y) - \hat{d}_{\mathbb{R}^{n-1}}(x,y) | \leq (8n \cdot \eta ^{1- (10n)^{-1}})^{2^{-(n-1)}} \leq 10^{-4} \eta ^{2^{-n}}.
\end{align}

\end{itemize}
Define $D_{\mathbf{k}}'':= \tilde{Z}(t_{\mathbf{k}}) \cap \mathbf{C}_{\mathbf{k}}(\delta _1)$, and $D_{\mathbf{k}}' := A_{\mathbf{k}, \delta _1}(t_{\mathbf{k}}) \cap \mathbf{C}_{\mathbf{k}}(\delta _1) \setminus \overline{D_{\mathbf{k}}''}$.
By the relative isoperimetric inequality and (\ref{small length k-n}), we know that
\begin{align}\label{aread''-n}
	\begin{split}
		\mathcal{H}^{n-1}(D_{\mathbf{k}}'') &\leq C(n)\cdot  \mathcal{H}^{n-2}(\tilde{\Sigma}^{n-2}(t_{\mathbf{k}}) ) ^{\frac{n-1}{n-2}}\\
					    &\leq C (n)\cdot  \mathcal{H}^{n-2}( \Sigma^{n-2}(t _{\mathbf{k}}) ) ^{ 1+ \frac{1- 10^{-4}}{n-2}}\\
					    &\leq C(n)\cdot \eta^{\frac{1}{4(n-2)}} \delta _0^{-1} \delta _1^{-1} \cdot  \mathcal{H}^{n-1}(\Sigma \cap \mathbf{C}_{\mathbf{k}}(\delta _1) )\\
				&\leq \mathcal{H}^{n-1}(\Sigma \cap \mathbf{C}_{\mathbf{k}}(\delta _1)).
\end{split}
\end{align}
We can take a finite covering $\{ B(x_\alpha , r_\alpha ): r_\alpha < \eta\} $ of $\overline{D_{\mathbf{k}}'}$ such that $\forall x \in B(x_\alpha , r_\alpha ) \cap D'_{\mathbf{k}}$, there is a line segment inside $D_{\mathbf{k}}'$ connecting $ x$ and $x_\alpha $. For any pair $x_\alpha , x_\beta $, let $\gamma _{\alpha , \beta } \subset \tilde{Y}(t_\mathbf{k}	) $ be an almost $\hat{d}_{\tilde{Y}(t_{\mathbf{k}})}$-geodesic between $x_\alpha$ and $x_\beta $.  Define $\tilde{D}_{\mathbf{k}}'$ by 
\begin{align*}
	\tilde{D}_{\mathbf{k}}' := D_{\mathbf{k}}' \cup \{\gamma _{\alpha , \beta }\} .
\end{align*}
In this way, using (\ref{Y-tld}), for any $x, y \in \tilde{D}_{\mathbf{k}}'$, we have a curve $\gamma \subset \tilde{D}_{\mathbf{k}}'$ connecting $x$ and $y$ such that
\begin{align}\label{dist-ind-n}
|\mathrm{Length}_{\mathrm{Eucl}}(\gamma ) - d_{\mathrm{Eucl}}(x,y)| \leq \eta + 10^{-4} \eta ^{2^{-n}} \leq 10^{-3}\eta ^{2^{-n}} .
\end{align}

For simplicity, we still denote $\tilde{D}_{\mathbf{k}}'$ by $D_{\mathbf{k}}'$.
 Let $\pi_{\mathbf{k}} : \mathbb{R}^n\to A_{\mathbf{k}, \delta _1}(t_{\mathbf{k}})$ be the orthogonal projection.

 Define 
 $$\mathbf{C}_{\mathbf{k}}(\delta _1)':= D_{\mathbf{k}}' \cup \left( \mathbf{C}_{\mathbf{k}}(\delta _1) \cap \pi_{\mathbf{k}}^{-1} (D_{\mathbf{k}}' \setminus \pi_\mathbf{k}(\Sigma \cap \mathbf{C}_\mathbf{k}(\delta _1) ) ) \right) .$$

 We then proceed by following the same arguments as presented in \cite[Section 4]{DS'23}. The main difference here is that we need to explicitly compute the quantitative upper bounds for the purpose of induction. Therefore, we will omit most of the proofs, providing them only when modifications are necessary.

\begin{lem}[{\cite[Lemma 4.3]{DS'23}}]\label{ooo-n}
	$\mathcal{H}^{n}( \mathbf{C}_{\mathbf{k}}(\delta _1)\setminus \mathbf{C}_{\mathbf{k}}(\delta _1)') \leq 2 \delta _1 \mathcal{H}^{n-1}(\Sigma \cap \mathbf{C}_{\mathbf{k}}(\delta _1) ) \leq 2\eta^{1+ 10^{-4}n^{-1}}.$	
\end{lem}

\begin{lem}[{\cite[Lemma 4.4]{DS'23}}]
	$\mathbf{C}_{\mathbf{k}}(\delta _1)'$ is path connected and $\mathbf{C}_{\mathbf{k}}(\delta _1)' \subset Y$.
\end{lem}	
	
Define 
	$$
	Y':= \cup _{\mathbf{k}\in \mathbb{Z}^n}\mathbf{C}_{\mathbf{k}}(\delta _1)' \subset Y.
	$$ 
	Notice that when $|\mathbf{k}|$ is big enough, one can certainly ensure that $\mathbf{C}_{\mathbf{k}}(\delta _1)'= \mathbf{C}_{\mathbf{k}}(\delta _1)$, so that $Y\setminus Y'$ is a bounded set.

	For any subset $V \subset Y$, let $V_t$ be the $t$-neighborhood of $V$ inside $(Y, \hat{d}_{\mathrm{Eucl},Y})$ in terms of the length metric $\hat{d}_{\mathrm{Eucl},Y}$, i.e.
	$$
	V_{t }:= \{y \in Y: \exists z \in V \text{ such that }\hat{d}_{\mathrm{Eucl},Y}(y, z) \leq t  \} .
	$$ 
	So $(Y')_t$ is the $t$-neighborhood of $Y'$ inside $(Y, \hat{d}_{\mathrm{Eucl}, Y})$.

	\begin{lem}[{\cite[Lemma 4.5]{DS'23}}] \label{ybis-n}
	There exists a closed cubset $Y''$ with smooth boundary such that $Y'\subset Y'' \subset (Y')_{6 \delta _0} $,
	$$
	\mathcal{H}^{n-1}(\partial Y'' ) \leq \delta _0^{-1} \mathcal{H}^{n-1}(\partial Y) \leq \eta^{1- 10^{-2}n^{-1}},
	$$ 
	and
	$Y''$ is contained in the $6\delta_0$-neighborhood of $Y'$ inside $Y''$, with respect to its length metric $\hat{d}_{\mathrm{Eucl}, Y''}$.
	
	%for any base point $q \in Y' $,
	%$$
	%d_{pGH}( ( Y', \hat{d}_{\mathrm{Eucl}, Y''}, q), (Y'' , \hat{d}_{\mathrm{Eucl},Y''}, q) ) \leq 6 \delta _0,
	%$$ 
	%and 
	%$$
	%d_{pGH}( (Y', d_{\mathrm{Eucl}}, q), (Y'', d_{\mathrm{Eucl}}, q) ) \leq 6 \delta _0.
	%$$ 

\end{lem}

Let $Y''$ be as in Lemma \ref{ybis-n}.
Recall that $\hat{d}_{\mathrm{Eucl},Y''}$ is defined as the length metric on $Y''$ induced by $g_{\mathrm{Eucl}}$. Since $Y' \subset Y'' \subset Y$, we have $d_{\mathrm{Eucl}} \leq \hat{d}_{\mathrm{Eucl},Y} \leq \hat{d}_{\mathrm{Eucl},Y''}$. 

\begin{lem}[{\cite[Lemma 4.6]{DS'23}}]\label{local small diam-n}
	$\diam_{\hat{d}_{\mathrm{Eucl}, Y''}}(\mathbf{C}_{\mathbf{k}}(\delta _1)' ) \leq (n+2) \delta _1 + 10^{-3}\eta ^{2^{-n}}.$
\end{lem}
\begin{proof}
	For any two points $x_1, x_2 \in \mathbf{C}_{\mathbf{k}}(\delta _1)'$, let $L_{x_1}, L_{x_2}$ be the line segments inside $\mathbf{C}_{\mathbf{k}}(\delta _1)$ through $x_1, x_2$ and orthogonal to $A_{\mathbf{k},\delta _1}(t_{\mathbf{k}})$ respectively. Let $x_1' = L_{x_1}\cap D_{\mathbf{k}}', x_2' = L_{x_2}\cap  D_{\mathbf{k}}'$. Then 
	by (\ref{dist-ind-n})
	we can find a curve $\gamma $ between $x_1', x_2'$ inside $D_{\mathbf{k}}'$ such that 
\begin{align*}
	\Length_{\mathrm{Eucl}}(\gamma ) &\leq d_{\mathrm{Eucl}}(x_1', x_2') + 10 ^{-3} \eta ^{2^{-n}}.
\end{align*}
	Consider the curve $\tilde{\gamma }$ consisting of three parts: the line segment $[x_1x_1']$ between $x_1, x_1'$, $\gamma $, and the line segment $[x_2'x_2]$ between $x_2', x_2$. We have $\tilde{\gamma } \subset \mathbf{C}_{\mathbf{k}}(\delta _1)' \subset Y'$, so
	$$
	\hat{d}_{\mathrm{Eucl}, Y''}(x_1, x_2) \leq L_{\mathrm{Eucl}}(\tilde{\gamma }) \leq (n+2) \delta _1 + 10^{-3}\eta ^{2^{-n}}.
	$$ 
\end{proof}

\begin{lem}[{\cite[Lemma 4.7]{DS'23}}] \label{ajout-n}
	For any base point $q \in Y'$ and any $D>0$,
	$$
	d_{pGH}( (Y' \cap B_{\mathrm{Eucl}}(q, D), \hat{d}_{\mathrm{Eucl},Y''},q), (Y' \cap B_{\mathrm{Eucl}}(q, D), d_{\mathrm{Eucl}},q) ) \leq 10^{-2}\eta ^{2^{-n}}.
$$ 
\end{lem}
\begin{proof}
	The same proof of \cite[Lemma 4.7]{DS'23} shows that  
	\begin{align*}
		d_{pGH}( (Y' \cap B_{\mathrm{Eucl}}(q, D), \hat{d}_{\mathrm{Eucl},Y''},q),& (Y' \cap B_{\mathrm{Eucl}}(q, D), d_{\mathrm{Eucl}},q) )\\
		&\leq 4 \cdot ( (n+2) \delta _1 + 10^{-3} \eta ^{2^{-n}} )\\
		&\leq 10^{-2}\eta ^{2^{-n}}.
	\end{align*}
\end{proof}
So we have:
\begin{prop}[{\cite[Proposition 4.8]{DS'23}}]\label{Y'-GH}
For any base point $q \in Y''$ and any $D>0$,
	$$
	d_{pGH}( (Y'' \cap B_{\mathrm{Eucl}}(q, D), \hat{d}_{\mathrm{Eucl},Y''},q), (Y''\cap B_{\mathrm{Eucl}}(q,D) , d_{\mathrm{Eucl}},q) ) \leq \frac{1}{4}\eta ^{2^{-n}}.
	$$ 
\end{prop}

For any $p \in Y'' $ and $D>0$, denote by $\hat{B}_{Y''}(p, D)$ the geodesic ball in $(Y'', \hat{d}_{\mathrm{Eucl}, Y'' })$, that is,
$$
\hat{B}_{Y''}(p,D):= \{x \in Y'' : \hat{d}_{\mathrm{Eucl}, Y'' }(p,x) \leq D\} .
$$ 

\begin{lem}[{\cite[Lemma 4.9]{DS'23}}]\label{Y'-B-B-hat-n}
For any base point $q \in Y''$ and any $D>0$,
	$$
	d_{pGH}( (Y'' \cap B_{\mathrm{Eucl}}(q, D), \hat{d}_{\mathrm{Eucl},Y''},q), ( \hat{B}_{\mathrm{Eucl},Y''}(q,D) , \hat{d}_{\mathrm{Eucl}, Y'' },q) ) \leq \frac{1}{4}\eta ^{2^{-n}}.
	$$ 
\end{lem}

To compare those metric spaces to the Euclidean $3$-space $(\mathbb{R}^3, g_{\mathrm{Eucl}})$, we need the following lemma, which is a corollary of the fact that $\mathcal{H}^{n-1}(\partial Y'')\leq \delta _1^{-1} \eta$.
\begin{lem}[{\cite[Lemma 4.12]{DS'23}}] \label{11111}
	For any $q \in Y'' $ and $D>0$,
	$$
	d_{pGH}((Y'' \cap B_{\mathrm{Eucl}}(q,D), d_{\mathrm{Eucl}}, q), (B_{\mathrm{Eucl}}(0, D), d_{\mathrm{Eucl}}, 0) ) \leq \frac{1}{4}\eta ^{2^{-n}}.
	$$ 
\end{lem}
\begin{proof}
	Under a translation diffeomorphism, we can assume $q=0$. It suffices to show that $B_{\mathrm{Eucl}}(q, D)$ lies in a $\frac{1}{4} \eta ^{2^{-n}}$-neighborhood of $Y'' $. If that were not the case, there would be a $\mu > \frac{1}{4} \eta ^{2^{-n}}$,
	and
	an $x \in B_{\mathrm{Eucl}}(q, D)$ with $B_{\mathrm{Eucl}}(x, \mu ) \cap Y'' = \emptyset$. 
But from the isoperimetric inequality, we have 
	$$
	\mathcal{H}^{n}(\mathbb{R}^3 \setminus Y'' )\leq C(n) \mathcal{H}^{n-1}(\partial Y'')^{\frac{n}{n-1}} \leq C(n) \eta ^{(1- 10^{-4}n^{-1})\cdot \frac{n}{n-1}},
	$$ 
	which would imply that
	$$
\omega _n\cdot  \frac{\eta ^{\frac{n}{2^{n}}}}{4^{n}} \leq	\omega _n \mu ^n=\mathcal{H}^{n}(B_{\mathrm{Eucl}}(x,\mu ) ) \leq C(n) \eta ^{\frac{n}{2(n-1)}},
	$$ 
	a contradiction when $\eta \leq \varepsilon (n) \ll 1$.
\end{proof}

This concludes the proof of the theorem.
\end{proof}
\begin{rmk}
	Note that in the proof above, since the underlying metric is Euclidean metric, we don't need to use \cite[Lemma 4.10]{DS'23}, where the estimate depends on a diameter upper bound. In general case when the underlying metric is $C^{0}$ close to Euclidean metric, we need to apply \cite[Lemma 4.10]{DS'23} to each small cube whose diameter is much smaller than $1$, so the conclusion still holds except that we need to replace the explicit error term (\ref{explicit error}) by $d_{GH}( \hat{B}_{g,Y''}(q, D) , B_{\mathrm{Eucl}}(0,D) ) \leq \Psi (\eta |n, D), \forall D>0$.
\end{rmk}

\bibliographystyle{alpha}
\bibliography{./math}

\newcommand{\etalchar}[1]{$^{#1}$}
\begin{thebibliography}{AMMO22}

\bibitem[ADM61]{ADM'61}
Richard Arnowitt, Stanley Deser, and Charles~W Misner.
\newblock Coordinate invariance and energy expressions in general relativity.
\newblock {\em Physical Review}, 122(3):997, 1961.

\bibitem[AMMO22]{AMMO22}
Virginia Agostiniani, Carlo Mantegazza, Lorenzo Mazzieri, and Francesca Oronzio.
\newblock {R}iemannian {P}enrose inequality via nonlinear potential theory.
\newblock {\em arXiv preprint arXiv:2205.11642}, 2022.

\bibitem[AMO24]{AMO21}
Virginia Agostiniani, Lorenzo Mazzieri, and Francesca Oronzio.
\newblock A {G}reen’s function proof of the positive mass theorem.
\newblock {\em Communications in Mathematical Physics}, 405(2):54, 2024.

\bibitem[Bar86]{Bartnik'86}
Robert Bartnik.
\newblock The mass of an asymptotically flat manifold.
\newblock {\em Communications on pure and applied mathematics}, 39(5):661--693, 1986.

\bibitem[BHK{\etalchar{+}}23]{BHKKZ23}
Hubert Bray, Sven Hirsch, Demetre Kazaras, Marcus Khuri, and Yiyue Zhang.
\newblock Spacetime harmonic functions and applications to mass.
\newblock In {\em Perspectives in scalar curvature}, pages 593--639. World Scientific, 2023.

\bibitem[BKKS22]{BKKS'22}
Hubert~L Bray, Demetre~P Kazaras, Marcus~A Khuri, and Daniel~L Stern.
\newblock Harmonic functions and the mass of 3-dimensional asymptotically flat {R}iemannian manifolds.
\newblock {\em The Journal of Geometric Analysis}, 32(6):184, 2022.

\bibitem[BL09]{BrayLee'09}
Hubert~L Bray and Dan~A Lee.
\newblock On the {R}iemannian {P}enrose inequality in dimensions less than eight.
\newblock {\em Duke Math. J.}, 146(1):81--106, 2009.

\bibitem[BM08]{BM08}
Hubert Bray and Pengzi Miao.
\newblock On the capacity of surfaces in manifolds with nonnegative scalar curvature.
\newblock {\em Inventiones mathematicae}, 172:459--475, 2008.

\bibitem[Bra97]{Bray97}
Hubert~Lewis Bray.
\newblock {\em The {P}enrose inequality in general relativity and volume comparison theorems involving scalar curvature}.
\newblock PhD thesis, stanford university, 1997.

\bibitem[Bra01]{Bray'01}
Hubert~L Bray.
\newblock Proof of the {R}iemannian {P}enrose inequality using the positive mass theorem.
\newblock {\em Journal of Differential Geometry}, 59(2):177--267, 2001.

\bibitem[Don24]{Dong22}
Conghan Dong.
\newblock Some stability results of positive mass theorem for uniformly asymptotically flat 3-manifolds.
\newblock {\em Annales math{\'e}matiques du Qu{\'e}bec}, pages 1--25, 2024.

\bibitem[DS25]{DS'23}
Conghan Dong and Antoine Song.
\newblock Stability of {E}uclidean 3-space for the positive mass theorem.
\newblock {\em Inventiones mathematicae}, 239(1):287--319, 2025.

\bibitem[EG15]{EG15}
Lawrence~Craig Evans and Ronald~F Gariepy.
\newblock {\em Measure Theory and Fine Properties of Functions, Revised Edition}.
\newblock Textbooks in Mathematics. CRC Press, 2015.

\bibitem[HI01]{HI'01}
Gerhard Huisken and Tom Ilmanen.
\newblock The inverse mean curvature flow and the {R}iemannian {P}enrose inequality.
\newblock {\em Journal of Differential Geometry}, 59(3):353--437, 2001.

\bibitem[KKL21]{KKL21}
Demetre Kazaras, Marcus Khuri, and Dan Lee.
\newblock Stability of the positive mass theorem under {R}icci curvature lower bounds.
\newblock {\em arXiv preprint arXiv:2111.05202}, 2021.

\bibitem[KWY17]{KWY17}
Marcus Khuri, Gilbert Weinstein, and Sumio Yamada.
\newblock Proof of the {R}iemannian {P}enrose inequality with charge for multiple black holes.
\newblock {\em Journal of Differential Geometry}, 106(3):451--498, 2017.

\bibitem[LNN23]{LNN20}
Man-Chun Lee, Aaron Naber, and Robin Neumayer.
\newblock $d_p$ convergence and $\epsilon$-regularity theorems for entropy and scalar curvature lower bounds.
\newblock {\em Geometry \& Topology}, 27(1):227--350, 2023.

\bibitem[LS12]{LS12}
Dan~A Lee and Christina Sormani.
\newblock Near-equality of the {P}enrose inequality for rotationally symmetric {R}iemannian manifolds.
\newblock {\em Annales Henri Poincare}, 13(7):1537--1556, 2012.

\bibitem[LS14]{LS14}
Dan~A. Lee and Christina Sormani.
\newblock Stability of the positive mass theorem for rotationally symmetric {R}iemannian manifolds.
\newblock {\em Journal f{\"u}r die reine und angewandte Mathematik (Crelles Journal)}, 2014(686):187--220, 2014.

\bibitem[MS15]{ManScho15}
Christos Mantoulidis and Richard Schoen.
\newblock On the {B}artnik mass of apparent horizons.
\newblock {\em Classical and Quantum Gravity}, 32(20):205002, 2015.

\bibitem[Ste22]{Stern'22}
Daniel~L Stern.
\newblock Scalar curvature and harmonic maps to {$S^1$}.
\newblock {\em Journal of Differential Geometry}, 122(2):259--269, 2022.

\bibitem[SY79]{SY79}
Richard Schoen and Shing-Tung Yau.
\newblock On the proof of the positive mass conjecture in general relativity.
\newblock {\em Communications in Mathematical Physics}, 65:45--76, 1979.

\bibitem[SY81]{SY'81}
Richard Schoen and Shing-Tung Yau.
\newblock Proof of the positive mass theorem. {II}.
\newblock {\em Communications in Mathematical Physics}, 79:231--260, 1981.

\bibitem[Wit81]{Witten81}
Edward Witten.
\newblock A new proof of the positive energy theorem.
\newblock {\em Communications in Mathematical Physics}, 80(3):381--402, 1981.

\end{thebibliography}

\end{document}